\documentclass[11pt]{amsart}
\usepackage[latin1]{inputenc}
\usepackage{amsmath}
\usepackage{amsfonts}
\usepackage{amssymb}
\usepackage{graphicx}
\usepackage{fourier}
\usepackage{dsfont}
\usepackage{hyperref}
\usepackage{enumerate}

\newtheorem{theorem}{Theorem}[section]
\newtheorem{lemma}[theorem]{Lemma}

\theoremstyle{definition}
\newtheorem{definition}[theorem]{Definition}

\theoremstyle{remark}
\newtheorem{remark}[theorem]{Remark}

\newcommand{\inner}[2]{\left\langle#1,#2\right\rangle}
\newcommand{\Norm}[1]{\left\lVert#1\right\rVert}
\newcommand{\abs}[1]{\left\lvert#1\right\rvert}
\newcommand{\pa}[1]{\left( #1 \right)}
\newcommand{\br}[1]{\left\lbrace #1\right\rbrace}

\numberwithin{equation}{section}

\author{Kleber Carrapatoso}
\address{K. Carrapatoso: Ceremade, Universit\'e Paris Dauphine\\
Place du Mar\'echal De Lattre De Tassigny\\
75775 Paris cedex 16 - France}
\email{carrapatoso@ceremade.dauphine.fr}
\author{Amit Einav$^1$}
\address{A. Einav: Department of Pure Mathematics and Mathematical Statistics, University of Cambridge\\
Wilberforce Road, Cambridge\\
CB3 0WB, United Kingdom.}
\email{A.Einav@dpmms.cam.ac.uk}
\title{Chaos and Entropic Chaos in Kac's Model Without High Moments.}
\thanks{Author $^1$ was supported by ERC grant MATKIT}

\begin{document}

\maketitle

\begin{abstract}
In this paper we present a new local L\'evy Central Limit Theorem, showing convergence to stable states that are not necessarily the Gaussian, and use it to find new and intuitive entropically chaotic families with underlying one-particle function that has moments of order $2\alpha$, with $1<\alpha<2$. We also discuss a lower semi continuity result for the relative entropy with respect to our specific family of functions, and use it to show a form of stability property for entropic chaos in our settings. 

\end{abstract}

\section{Introduction}\label{sec: introduction}
One of the most important equation in the kinetic theory of gases, describing the evolution in time of the distribution function of a dilute gas, is the so-called Boltzmann equation. In its spatially homogeneous form it reads as
\begin{equation}\label{eq: boltzmann equation}
\begin{gathered}
\frac{\partial f}{\partial t}(v)=Q(f,f)(v), \quad v\in\mathbb{R}^d,\;\;t >0\\
f|_{t=0}=f_0,
\end{gathered}
\end{equation}
where $d\geq 2$ and $Q$ is the quadratic Boltzmann collision operator, given by
\begin{equation}\label{eq: boltzmann collision operator Q}
 Q(f,g)=\int_{\mathbb{R}^d \times \mathbb{S}^{d-1}}B\pa{\abs{v-v_\ast},\cos\pa{\theta}}\pa{f^\prime g^\prime_{\ast}-fg_\ast}d\sigma dv_{\ast}.
\end{equation}
We have used the notations $f^\prime (v)= f(v^\prime)$, $f_\ast(v)=f(v_\ast)$ and $f^\prime_\ast(v)=f(v_\ast^\prime)$ with
\begin{equation}\nonumber
v^\prime=\frac{v+v_\ast}{2}+\frac{\abs{v-v_\ast}}{2}\sigma, \;\; v^\prime_\ast=\frac{v+v_\ast}{2}-\frac{\abs{v-v_\ast}}{2}\sigma.
\end{equation}
representing the pre-collision velocities of particles with post-collision velocities $v,v_\ast$. The above relationships are a direct result of conservation of momentum and energy for the associated problem. The function $B$, called the Boltzmann collision kernel, is determined by the physics of the problem (mainly via the collisional cross-section) and it is assumed that $B$ is non-negative and depends only on the magnitude of the relative velocity, $\abs{v-v_\ast}$, and the cosine of the deviation angle between $v-v_\ast$ and $v^\prime-v^\prime_\ast$, $\theta \in [0,\pi]$.\\
In his work on equation (\ref{eq: boltzmann equation}), Boltzmann investigated the concept of the entropy and gave an interpretation to it in the microscopic setting, as well as a formula to it in terms of the distribution function $f$:
\begin{equation}\label{eq: entropy}
H(f)= \int_{\mathbb{R}^d}f(v)\log f(v) dv.
\end{equation}
One important contribution Boltzmann made was his famous $H-$Theorem: Under the evolution of (\ref{eq: boltzmann equation}) one has that
\begin{equation}\label{eq: h-theorem}
D(f) = -\frac{d}{dt}H(f) = \int_{\mathbb{R}^d }Q(f,f)(v)\log f(v) dv \geq 0,
\end{equation}
where $D(f)$, called \emph{the entropy production}, is defined as the minus of the formal derivative of the entropy under the evolution of the Boltzmann equation. One can easily see that functions of the form
\begin{equation}\nonumber
M_a(v)=\frac{e^{-\frac{\abs{v}^2}{2a}}}{\pa{2\pi a}^{\frac{d}{2}}},
\end{equation}
where $a>0$, satisfy $Q\pa{M_a,M_a}=0$, and as such present a stationary solution to (\ref{eq: boltzmann equation}) that is a critical point to the entropy functional. Such functions are usually called Maxwellians and represent the equilibrium states of the Boltzmann equation. One would hope that under suitable conditions we will gain convergence to equilibrium in our equation. This problem has been investigated by many authors, starting with Carleman and continuing to this day.\\
There are two fundamental questions in Kinetic Theory that pertain to the spatially homogeneous Boltzmann equation:
\begin{enumerate}[1.]
\item One of the main problems with the Boltzmann process is its \emph{irreversibility}. The reason behind this is Boltzmann's 'Stosszahlansatz' assumption that pre collisional particles can be considered to be independent. However, a closed system like that of dilute gas should obey Poincar\'e's recurrence principle and eventually come back to its original state. How can the equation be valid in that case? The answer to this, given by Boltzmann himself, is in the time scale. The Boltzmann equation, and trend to equilibrium, can only be valid in a time scale that is much smaller than the time it'll take the system to return to its original state. As such, the question of finding a \emph{quantitative} rate of convergence to equilibrium to the Boltzmann equation is of paramount importance.
\item While used in practice there is no full proof that is valid for times in the macroscopic scale, of how one can get the Boltzmann equation from reversible Newtonian laws. This, too, is a very important problem in Kinetic Theory. The best result attained so far is one by Lanford, \cite{Lanford}, in 1975. One possible intuition of how one can get an irreversible process from reversible laws lies with adding probability into the mixture. Either via randomness in the spatial variable, or via a many-particle model form which the Boltzmann equation arise as a mean field limit. While we will mainly focus on the latter option, we'd like to mention that there are other possibilities for the rise of such processes, such as loss of regularity and coarse graining at the microscopic level. 
\end{enumerate}
In his 1956 paper, \cite{Kac}, Kac attempted to give a partial solution to these two problems. Kac introduced a many-particle model, consisting of $N$ indistinguishable particle with one dimensional velocities, undergoing binary collision and constrained to the energy sphere $\mathbb{S}^{N-1}\pa{\sqrt{N}}$, which we will call 'the Kac's sphere' from this point onward. Kac's evolution equation is given by
\begin{equation}\label{eq: kac's evolution equation}
\frac{\partial F_N}{\partial t}\pa{v_1,\dots,v_N}= -N(I-Q)F_N\pa{v_1,\dots,v_N},
\end{equation}
where $F_N$ represents the distribution function of the $N$ particles, and the gain term $Q$ is given by
\begin{equation}\label{eq: gain term Q}
\begin{gathered}
QF\pa{t,v_1,\dots,v_N}=\frac{1}{2\pi}\frac{2}{N(N-1)}\sum_{i<j} \\
\int_0^{2\pi}F\pa{t,v_1,\dots,v_i(\theta),\dots,v_j(\theta),\dots,v_N}d\theta,
\end{gathered}
\end{equation}
with 
\begin{equation}\label{eq: v(theta)}
v_i(\theta)=v_i \cos (\theta) + v_j\sin(\theta), \quad v_j(\theta)=-v_i \sin(\theta) + v_j \cos(\theta).
\end{equation}
Motivated by Boltzmann's pre-collisional assumption, Kac defined the concept of \emph{Chaoticity} (what he called 'the Boltzmann property' in his paper) as follows:
\begin{definition}\label{def: chaoticty for functions}
A symmetric family of distribution functions $\br{F_N}_{N\in\mathbb{N}}$ on Kac's sphere is called \emph{chaotic} if there exists a distribution function on $\mathbb{R}$, $f$, such that for any $k\in\mathbb{N}$
\begin{equation}\label{eq: chaoticity for functions}
\lim_{N\rightarrow\infty}\Pi_k\pa{F_N}\pa{v_1,\dots,v_k}=f^{\otimes k}\pa{v_1,\dots,v_k},
\end{equation}
where $\Pi_k\pa{F_N}$ is the $k-$marginal of $F_N$, and the limit is taken in the weak topology induced by bounded continuous functions.
\end{definition}
Using a beautiful combinatorial argument, Kac showed that the property of chaoticity propagates with his evolution equation, i.e. if $\br{F_N\pa{0,v_1,\dots,v_N}}_{N\in\mathbb{N}}$ is $f_0-$chaotic then the solution to equation (\ref{eq: kac's evolution equation}), $\br{F_N\pa{t,v_1,\dots,v_N}}_{N\in\mathbb{N}}$ is $f_t-$chaotic, where $f_t$ solves a caricature of the Boltzmann equation: 
\begin{equation}\label{eq: caricature boltzmann}
\frac{\partial f}{\partial t}(v) = \frac{1}{2\pi}\int_{\mathbb{R}}\int_0^{2\pi}\pa{f\pa{v(\theta)}f\pa{v_\ast(\theta)}-f\pa{v}f\pa{v_\ast}}dv_\ast d\theta,
\end{equation}
with $v\pa{\theta}, v_\ast\pa{\theta}$ given by (\ref{eq: v(theta)}).
While we only got a distorted form of the Boltzmann equation, with collision kernel $B\equiv 1$, the ideas presented in Kac's paper were powerful enough that McKean managed to extend them to the $d-$dimensional case (see \cite{McKean}). Under similar condition to those presented by Kac, McKean construct a similar $N-$particle model from which the \emph{real} spatially homogeneous Boltzmann equation arose as mean field limit for certain collision kernels (mainly those who are independent of the relative velocity). We will not discuss this model in this work, and refer the interested reader to \cite{CGL, Einav2, McKean} for more information.\\
Giving a partial answer to the validation of the Boltzmann equation, Kac set out to try and find a partial solution to the rate of convergence as well. He noticed that his evolution equation is ergodic, with an equilibrium state represented by the constant function $1$. As such, for any fixed $N$, one can easily see that
\begin{equation}\nonumber
\lim_{t\rightarrow\infty}F_N\pa{t,v_1,\dots,v_N}=1.
\end{equation}
The rate of convergence to equilibrium is determined by the spectral gap
\begin{equation}\nonumber
\Delta_N = \inf \br{ \frac{\inner{\varphi}{N(I-Q)\varphi}_{L^2\pa{\mathbb{S}^{N-1}\pa{\sqrt{N}}}}}{\Norm{\varphi}^2_{L^2\pa{\mathbb{S}^{N-1}\pa{\sqrt{N}}}}}\; | \; \varphi\text{ is symmetric },\varphi\in L^2\pa{\mathbb{S}^{N-1}\pa{\sqrt{N}}}, \; \varphi\perp 1}.
\end{equation}
Kac's conjectured that
\begin{equation}\nonumber
\Delta = \liminf_{N\rightarrow\infty}\Delta_N >0,
\end{equation}
which would lead to 
\begin{equation}\label{eq: spectral gap rate of convergence}
\Norm{F_N\pa{t,\cdot}-1}_{L^2\pa{\mathbb{S}^{N-1}\pa{\sqrt{N}}}} \leq e^{-\Delta t}\Norm{F_N\pa{0,\cdot}-1}_{L^2\pa{\mathbb{S}^{N-1}\pa{\sqrt{N}}}}.
\end{equation}
The spectral gap problem remained open until $2000$, when a series of papers by authors such as Janversse, Maslen, Carlen, Carvahlo, Loss and Geronimo gave a satisfactory positive answer to the conjecture, even in McKean's model (see \cite{Jan,Maslen,CCL,CGL} for more details). However, the $L^2$ norm is catastrophic when dealing with chaotic families. In that case, attempts to pass to the limit in the number of particles is futile. \\
Taking lead from the real Boltzmann equation, one can define the entropy on Kac's sphere as 
\begin{equation}\label{eq: entropy}
H_N(F_N)=\int_{\mathbb{S}^{N-1}\pa{\sqrt{N}}}F_N \log F_N d\sigma^N,
\end{equation}
where $d\sigma^N$ is the uniform probability measure on Kac's sphere. The reason behind this choice is the \emph{extensivity} property of the entropy: In a very intuitive way, we'd like to think that 'nice' $f-$chaotic families behave like $F_N\approx f^{\otimes N}$, as such
\begin{equation}\label{eq: entropic chaoticity intuition}
H_N(F_N)\approx N\int_{\mathbb{R}}f(v)\log \pa{\frac{f(v)}{\gamma(v)}}dv,
\end{equation} 
 where $\gamma$ is the standard Gaussian on $\mathbb{R}$. This intuition was defined formally in \cite{CCRLV}, where the authors investigated the entropy functional on the Kac's sphere:
\begin{definition}\label{def: entropic chaoticity on the sphere} 
 An $f-$chaotic family of distribution functions on the sphere is called \emph{entropically chaotic} if
\begin{equation}\label{eq: entropic chaoticity for functions}
\lim_{N\rightarrow\infty}\frac{H_N(F_N)}{N}=\int_{\mathbb{R}}f(v)\log \pa{\frac{f(v)}{\gamma(v)}}dv=H(f |\gamma).
\end{equation} 
\end{definition}
The concept of entropic chaoticity is much stronger than that of chaoticity as it involves the correlation between arbitrary number of particles. We will verify this intuition later on in this paper.\\
Defining the \emph{entropy production} to be the minus of the formal derivation of the entropy under Kac's evolution equation
\begin{equation}\label{eq: entropy production}
 D_N(F_N)=-\frac{d}{dt}H_N(F_N)=\inner{\log F_N}{N(I-Q)F_N}_{L^2 \pa{\mathbb{S}^{N-1}\pa{\sqrt{N}}}},
\end{equation}
one can define the appropriate 'spectral gap' by 
\begin{equation}\nonumber
\Gamma_N = \inf_{F_N}\frac{D_N(F_N)}{H_N(F_N)}
\end{equation}
and ask if there is a positive constant, $C>0$, such that 
\begin{equation}\nonumber
\Gamma_N \geq C
\end{equation}
for all $N$. This problem is the so-called many-body Cercignani's conjecture, named after a similar conjecture posed for the real Boltzmann equation in \cite{Cerc}. If there exists such a $C$, we have that
\begin{equation}\nonumber
H_N(F_N(t)) \leq e^{-Ct}H_N(F_N(0)).
\end{equation}
Combining this with equation (\ref{eq: entropic chaoticity for functions}) and taking the limit as $N$ goes to infinity, one can hope to get that
\begin{equation}\label{eq: rate of convergence with cercignani's}
H(f_t|\gamma)\leq e^{-Ct}H(f_0 |\gamma).
\end{equation}
This, along with a known inequality on $H(f|\gamma)$ gives an exponential rate of decay towards the equilibrium. \\
Unfortunately, in general, Cercignani's many body conjecture is false. We will discuss this shortly, as it motivates part of the presented work. We refer the reader to \cite{CCRLV,Einav1,Einav2} for more information about this.\\
At this point, the reader might ask whether or not chaotic states exist, and whether the intuition $F_N \approx f^{\otimes N}$ is reasonable. The answer to both questions is Yes. We start by constructing a chaotic family following this exact intuition:\\
Given a distribution function $f$ on $\mathbb{R}$, we define 
\begin{equation}\label{eq: usual suspect}
F_N\pa{v_1,\dots,v_N}=\frac{f^{\otimes N}\pa{v_1,\dots.v_N}}{\mathcal{Z}_N \pa{f,\sqrt{N}}},
\end{equation}
where \emph{the normalisation function}, $\mathcal{Z}_N\pa{f,r}$ is defined by
\begin{equation}\label{eq: normalissation function}
\mathcal{Z}_N\pa{f,r} = \int_{\mathbb{S}^{N-1}(r)}f^{\otimes N}d\sigma^N_r,
\end{equation}
with $d\sigma^N_r$ the uniform probability measure on $\mathbb{S}^{N-1}(r)$. Kac himself discussed such functions, and have shown that they are chaotic when $f$ has very strong integrability conditions. In a recent paper by Carlen, Carvahlo, Le Roux, Loss and Villani, \cite{CCRLV}, the authors have managed to extend Kac's result to the following:
\begin{theorem}\label{thm: usual suspects are good}
Let $f$ be a probability density on $\mathbb{R}$ such that $f\in L^p(\mathbb{R})$ for some $p>1$, $\int_\mathbb{R} x^2 f(x)=1$ and $\int_\mathbb{R} x^4 f(x)dx<\infty$. Then the family of densities defined in (\ref{eq: usual suspect}) is $f-$chaotic. Moreover, it is $f-$entropically chaotic.
\end{theorem}
The main tool to prove Theorem \ref{thm: usual suspects are good} is a local central limit theorem, giving an approximation for the normalisation function, $\mathcal{Z}_N\pa{f,\sqrt{r}}$: 
\begin{theorem}\label{thm: normalisation function fixed f approximation}
Let $f$ satisfy the conditions of Theorem \ref{thm: usual suspects are good}, then 
\begin{equation}\label{eq: normalisation function fixed f approximation}
\mathcal{Z}_N(f,\sqrt{u})=\frac{2}{\sqrt{N}\Sigma \left\lvert \mathbb{S}^{N-1}\right\rvert  u^{\frac{N-2}{2}}}\left( \frac{e^{-\frac{(u-N)^2}{2N\Sigma^2}}}{\sqrt{2\pi}}+\lambda_N(u) \right),
\end{equation}
where $\Sigma^2 = \int_{\mathbb{R}}v^4f(v)dv - 1$ and $\sup_u \abs{\lambda_N(u)}\underset{N\rightarrow\infty}{\longrightarrow}0$.
\end{theorem}
We'd like to mention at this point that the above theorems were extended to to McKean's model by the first author in \cite{Carr}.\\
In \cite{Einav1} (and later on in \cite{Einav2} for McKean's model) the second author extended the above local central limit theorem to the case where the underlying generating function, $f$, also varies with $N$:
\begin{theorem}\label{thm: einav1}
Let $0<\eta<1$ and $\delta_N=\frac{1}{N^\eta}$. Define
\[
f_N(v)=\delta_N M_{\frac{1}{2\delta_N}}(v)+(1-\delta_N)M_{\frac{1}{2(1-\delta_N)}}(v),
\]
where $M_a(v)=\frac{e^{-\frac{v^2}{2a}}}{\sqrt{2\pi a}}$. Then 
\begin{equation}\label{eq: normalisation function f_N approximation}
\mathcal{Z}_N(f,\sqrt{u})=\frac{2}{\sqrt{N}\Sigma_N \left\lvert \mathbb{S}^{N-1}\right\rvert  u^{\frac{N-2}{2}}}\left( \frac{e^{-\frac{(u-N)^2}{2N\Sigma_N^2}}}{\sqrt{2\pi}}+\lambda_N(u) \right),
\end{equation}
where $\Sigma^2 = \frac{3}{4\delta_N(1-\delta_N)}-1$ and $\sup_u \abs{\lambda_N(u)}\underset{N\rightarrow\infty}{\longrightarrow}0$. Moreover, using the same notation as (\ref{eq: usual suspect}) 
with $f$ replaced by $f_N$, one finds that there exists $C_{\eta^\prime}>0$, depending only on $\eta^\prime$ such that
\begin{equation}\label{eq: cercignani's conjecture is false}
\Gamma_N \leq \frac{D_N(F_N)}{H_N(F_N)} < \frac{C_{\eta^\prime}}{N^{\eta^\prime}}
\end{equation}
for $0<\eta^\prime<\eta$.
\end{theorem}
The above theorem shows exactly how high order moments play an important role in the evaluation of the minimal entropy-entropy production ratio, $\Gamma_N$. The family constructed in Theorem \ref{thm: einav1} has two peculiar properties:
\begin{enumerate}[(i)]
\item \begin{equation}\nonumber
\int_{\mathbb{R}}v^4 f_N(v)dv = \frac{3}{4\delta_N(1-\delta_N)}\underset{N\rightarrow\infty}{\longrightarrow}\infty.
\end{equation}
\item One can check that $F_N$ is $M_{\frac{1}{2}}-$chaotic yet $\lim_{N\rightarrow\infty}\frac{H_N(F_N)}{N}$ exists but doesn't equal $H\pa{M_{\frac{1}{2}}|\gamma}$! 
\end{enumerate}
 Will the many-body Cercignani's conjecture be true if we restrict ourselves to families that violates $(i)$ and $(ii)$? is there a connection between $(i)$ and $(ii)$?\\
 Motivated by the above questions, we set out to investigate the effects of the fourth moment on chaoticity and entropic chaoticity. We consider families of distribution functions on Kac's sphere, $\br{F_N}_{N\in\mathbb{N}}$, of the form (\ref{eq: usual suspect}) where the underlying generating function $f$ is independent of $N$, but has moment of order $2\alpha$, with $1<\alpha<2$. Surprisingly enough, a lot can be said about this case, much like the case where $f$ has a fourth moment.\\
Before we state our main results, we will extend the definition of chaoticity and entropic chaos to general symmetric measures on Kac's sphere, as well as define the relative entropy and the relative Fisher information functional.
\begin{definition}
Given two probability measures, $\mu,\nu$, on a Polish space $X$, we define the relative entropy $H(\mu | \nu)$ as
\begin{equation}\label{def: relative entropy }
H(\mu|\nu) = \int_X h\log h d\nu,
\end{equation}
where $h=\frac{d\mu}{d\nu}$, and $H(\mu|\nu)=\infty$ if $\mu$ is not absolutely continuous with respect to $\nu$.
\end{definition}  
Notice that in our notations
\begin{equation}\nonumber
H_N(F_N)=H(F_N d\sigma^N | d\sigma^N),
\end{equation}
with an underlying space $X=\mathbb{S}^{N-1}\pa{\sqrt{N}}$.
\begin{definition}\label{def: symmetric measure}
A probability measure on a space $X$ that is invariant under the action of the symmetric group $\mathcal{S}_N$ is called symmetric if
\begin{equation}\label{eq: symmetric measure}
\int_X fd\mu = \int_X f\circ \tau d\mu,
\end{equation}
for any $\tau\in\mathcal{S}_N$, and $f\in C_b(X)$.
\end{definition}
\begin{definition}\label{def: chaoticity of measure}
A family of symmetric probability measures on Kac's sphere, $\br{\mu_N}_{N\in\mathbb{N}}$, is called $\mu-$chaotic, where $\mu$ is a probability measure on $\mathbb{R}$, if for any $k\in\mathbb{N}$
\begin{equation}\label{eq: def of chaoticity for measures}
\lim_{N\rightarrow\infty}\Pi_k\pa{\mu_N}=\mu^{\otimes k},
\end{equation}
where $\Pi_k(\mu_N)$ is the $k-$th marginal of $\mu_N$ and the limit is in the weak topology.
\end{definition}
It is a known result (see \cite{Sznitman} for instance) that it is enough to check the marginals for $k=1,2$ in order to conclude chaoticity.
\begin{definition}\label{def: def of entropic chaoticity for measures}
A symmetric $\mu-$chaotic family of probability measures on Kac's sphere, $\br{\mu_N}_{N\in\mathbb{N}}$, is called entropically chaotic if 
\begin{equation}\label{eq: def of entropci choticity for measures}
\lim_{N\rightarrow\infty}\frac{H_N\pa{\mu_N | d\sigma^N}}{N}=H\pa{\mu | \gamma},
\end{equation}
where $d\sigma^N$ is the uniform probability measure on Kac's sphere and $H(\mu | \gamma) $ is the relative entropy of $\mu$ and $\gamma(v)dv$.
\end{definition}
Lastly we define the relative Fisher information functional, which has intimate relation to the relative entropy. As it requires a lot more information on the space on which the measures act, we define it only on $\mathbb{R}$, and Kac's Sphere:
\begin{definition}\label{def: relative fisher transform}
Given two probability measures, $\mu,\nu$ on $\mathbb{R}$, we define the relative Fisher information functional $I(\mu | \nu)$ as
\begin{equation}\label{def: relative fisher information}
I(\mu|\nu) = \int_{\mathbb{R}} \frac{\abs{h^\prime (x)}^2}{h(x)} d\nu(x) = 4\int_{\mathbb{R}}\abs{\frac{d}{dx}\sqrt{h(x)}}^2d\nu(x),
\end{equation}
where $h=\frac{d\mu}{d\nu}$, and $I(\mu|\nu)=\infty$ if $\mu$ is not absolutely continuous with respect to $\nu$.\\
Given two probability measures, $\mu_N,\nu_N$ on Kac's sphere, we define the relative Fisher information functional $I_N(\mu_N | \nu_N)$ as
\begin{equation}\label{def: relative fisher information}
I_N(\mu_N|\nu_N) = \int_{\mathbb{S}^{N-1}\pa{\sqrt{N}}} \frac{\abs{\nabla_S h}^2}{h} d\nu,
\end{equation}
where $h=\frac{d\mu_N}{d\nu_N}$, and $I_N(\mu_N|\nu_N)=\infty$ if $\mu_N$ is not absolutely continuous with respect to $\nu_N$. Here $\nabla_S$ denotes the components of the usual gradient on $\mathbb{R}^N$ that is tangential to Kac's sphere.
\end{definition}
The main results of our paper are as follows:
\begin{theorem}\label{thm: chaoticity and entropic chaoticity general}
Let $f$ be a probability density such that $f\in L^{p}$ for some $p>1$ and $\int x^2 f(x)dx=1$. Let 
\begin{equation}\label{eq: definition of mu_2 in thm}
\nu_f(x)=\int_{-\sqrt{x}}^{\sqrt{x}}y^4 f(y)dy
\end{equation} 
and assume that $\nu_f(x)\underset{x\rightarrow\infty}{\sim} C_{S} x^{2-\alpha}$ for some $C_{S}>0$ and $1<\alpha<2$. Then the family
\begin{equation}\nonumber
F_N=\frac{f^{\otimes N}}{\mathcal{Z}_N\pa{f,\sqrt{N}}}
\end{equation} 
is $f-$chaotic. Moreover, it is $f-$entropically chaotic.
\end{theorem}
In particular, one has that 
\begin{theorem}\label{thm: chaoticity and entropic chaoticity with growth conditions on f}
Let $f$ be a probability density such that $f\in L^{p}$ for some $p>1$ and $\int x^2 f(x)dx=1$. Assume in addition that 
\begin{equation}\label{eq: growth conditions on f}
f(x)\underset{x\rightarrow\infty}{\sim} \frac{D}{|x|^{1+2\alpha}},
\end{equation} 
for some $1<\alpha<2$ and $D>0$. Then the family
\begin{equation}\nonumber
F_N=\frac{f^{\otimes N}}{\mathcal{Z}_N\pa{f,\sqrt{N}}}
\end{equation} 
is $f-$chaotic. Moreover, it is $f-$entropically chaotic.
\end{theorem}
In addition to the above, the probability measure $\nu_N=F_Nd\sigma^N$ plays an important role on Kac's sphere. This is expressed in the following distorted lower semi continuity property:
\begin{theorem}\label{thm: lower semi continuity}
Let $f$ satisfy the conditions of Theorem \ref{thm: chaoticity and entropic chaoticity general} and let $\mu_N$ be a symmetric probability measure on Kac's sphere such that for some $k\in\mathbb{N}$
\begin{equation}\label{eq: weak convergence per k}
\Pi_{k}(\mu_N) \underset{N\rightarrow\infty}{\rightharpoonup}\mu_{k},
\end{equation}
where $\mu_{k}$ is a probability measure on $\mathbb{R}^{k}$. Then, if we denote by $F_N=\frac{f^{\otimes N}}{\mathcal{Z}_N (f,\sqrt{N})}$ and $\nu_N=F_Nd\sigma^N$ we have that
\begin{enumerate}[(i)]
\item $\Pi_{1}(\mu_N) \underset{N\rightarrow\infty}{\rightharpoonup}\Pi_1\pa{\mu_{k}}=\mu$ and 
\begin{equation}\label{eq: l.s.c. for k=1}
H(\mu | f) \leq \liminf_{N\rightarrow\infty}\frac{H_N(\mu_N|\nu_N)}{N},
\end{equation}
where $H(\mu | f)$ is the relative entropy between $\mu$ and the measure $f(v)dv$.
\item For any $\delta>0$ we have that
\begin{equation}\label{eq: semi lsc for k>1}
\begin{gathered}
\liminf_{N\rightarrow\infty}\frac{H(\mu_N|\nu_N)}{N}\geq \frac{H(\mu_k|f^{\otimes k})}{k}-\limsup_{N\rightarrow\infty}\int_{\mathbb{R}} \log \left(f(v)+\delta \right)d\Pi_1 \left(\mu_N \right)(v) \\
+ \int \log\left(f(v) \right)d\mu(v)  -\frac{1-\int |v|^2 d\mu(v)}{2}.
\end{gathered}
\end{equation}
\end{enumerate}
\end{theorem}
Theorem \ref{thm: lower semi continuity} is the key to proving the following stability property of entropic chaoticity:
\begin{theorem}\label{thm: stability property}
Let $f$ satisfy the conditions of Theorem \ref{thm: chaoticity and entropic chaoticity general} and assume in addition that $f\in L^\infty (\mathbb{R})$. Then, if 
\begin{equation}\label{eq: entropic closeness}
\lim_{N\rightarrow\infty}\frac{H(\mu_N | \nu_N)}{N}=0,
\end{equation}
where $\nu_N$ was defined in Theorem \ref{thm: lower semi continuity}, $\mu_N$ is $f-$chaotic. Moreover, $\mu_N$ is $f-$entropically chaotic.
\end{theorem}
A different approach to the stability problem involves the relative Fisher information functional on Kac's sphere, $I_N$:
\begin{theorem}\label{thm: stability property via fisher}
Let $\br{\mu_N}_{N\in\mathbb{N}}$ be a family of symmetric probability measures on Kac's sphere that is $f-$chaotic. Assume that there exists $C_S>0$ and $1<\alpha<2$ such that
\begin{equation}\label{eq: uniform nu_{mu_N}}
\int_{-\sqrt{x}}^{\sqrt{x}}v_1^4 d\Pi_1(\mu_N)(v_1) \underset{x\rightarrow\infty}{\sim}C_S x^{2-\alpha}
\end{equation}
uniformly in $N$, and that 
\begin{equation}\label{eq: condition of stability with fisher II}
\begin{gathered}
\frac{H_N(\mu_N | \sigma^N)}{N} \leq C, \quad \frac{I_N(\mu_N | \sigma^N)}{N}\leq C 
\end{gathered}
\end{equation}
for all $N$ and some $2<k<4$. Then $\mu_N$ is $f-$entropically chaotic.
\end{theorem}


The presented work is structured as follows: In Section \ref{sec: preliminaries} we will present some preliminaries to the work, including known results on the normalisation function, marginals of probability measures on Kac's sphere and stable $\alpha$ processes. Section \ref{sec: local central limit} will be focused on finding a local L\'evy Central Limit Theorem, to be used in Section \ref{sec: chaoticity and entropic chaoticity}, where we will prove Theorems \ref{thm: chaoticity and entropic chaoticity general} and \ref{thm: chaoticity and entropic chaoticity with growth conditions on f}. In Section \ref{sec: lsc and stability} we will discuss the lower semi continuity property of processes of our type (Theorem \ref{thm: lower semi continuity}) and prove the stability theorems, Theorems \ref{thm: stability property} and \ref{thm: stability property via fisher}. Section \ref{sec: final remarks} will see closing remarks for our work, while the Appendix will discuss a quantitative L\'evy type approximation theorem, and include some additional computation that would otherwise encumber the presentation of our paper.\\
\\
For more information about the Boltzmann equation, Kac's (and McKean's) model, the spectral gap and entropy-entropy production problems, as well as discussion about chaoticity and entropic chaoticity we refer the interested reader to \cite{CCL,CCL1,CCRLV,CGL,Carr,Einav1,Einav2,Einav3,HM,MM,Villani,VReview}.\\
\\
\textbf{Acknowledgement.} The authors would like to thank Cl\'ement Mouhot and St\'epahne Mischler for fruitful discussions, constant encouragements and support, as well as enlightening remarks about the manuscript. 

\section{Preliminaries.}\label{sec: preliminaries}
\subsection*{The Normalisation Function.} As discussed in the introduction, the normalisation function, $\mathcal{Z}_N\pa{f,\sqrt{r}}$, plays an important role in the proofs of chaoticity and entropic chaoticity of distribution families of the form 
\begin{equation}\nonumber
F_N=\frac{f^{\otimes N}}{\mathcal{Z}_N\pa{f,\sqrt{N}}}.
\end{equation}
In this short subsection we will give a probabilistic interpretation to it, as well as explain why it is well defined under simple conditions on $f$.\\
Before we begin, we'd like to make a small remark about notation convention: we frequently use the term 'distribution function' in this paper, by which we mean the Statistical Physics sense of the term, i.e. a probability density function in mathematical terms. In what follows, when we'll aim to be very precise and less confusing, we'll use the terms 'probability density function' and 'probability distribution function' to clarify certain conditions of theorems. 
\begin{lemma}\label{lem: interpretation of normalisation function}
Let $f$ be a probability density function for the real random variable $V$. Then
\begin{equation}\label{eq: interpretation of the normalisation function}
\mathcal{Z}_N\pa{f,\sqrt{r}}=\frac{2h^{\ast N}(r)}{\abs{\mathbb{S}^{N-1}}r^{\frac{N-2}{2}}}
\end{equation}
where $h$ be the associated probability density function for the real random variable $V^2$ and $h^{\ast N}$ is the $N-$th iterated convolution of $h$. \\
\end{lemma} 
Proof for the above lemma can be found in \cite{CCRLV,Einav1}, yet we present it here for completion.
\begin{proof}
Denote by $S_N=\sum_{i=1}^{N} V_i^2$ the sum of independent copies of the real random variable $V^2$. 
For any function $\varphi\in C_b\pa{\mathbb{R}^N}$, depending only on $r=\sqrt{\sum_{i=1}^N v_i^2}$ we find that
\begin{equation}\nonumber
\begin{gathered}
\mathds{E}\varphi=\int_{\mathbb{R}^N}\varphi\pa{\sqrt{\sum_{i=1}^N v_i^2}}\Pi_{i=1}^N f(v_i)dv_1\dots dv_N =\\
 \abs{\mathbb{S}^{N-1}}\int_0^\infty \varphi(r)r^{N-1}\pa{\int_{\mathbb{S}^{N-1}(r)}\Pi_{i=1}^N f(v_i)d\sigma^N_r}dr= \abs{\mathbb{S}^{N-1}}\int_0^\infty \varphi(r)r^{N-1}\mathcal{Z}_N\pa{f,r}dr
\end{gathered}
\end{equation}
On the other hand
\begin{equation}\nonumber
\begin{gathered}
\mathds{E}\varphi = \int_{0}^\infty \varphi \pa{\sqrt{r}}s_N(r)dr=2\int_{0}^\infty r\varphi(r) s_N\pa{r^2}dr.
\end{gathered}
\end{equation}
Since the above is valid for any $\varphi$ we conclude that
\begin{equation}\nonumber
\mathcal{Z}_N\pa{f,\sqrt{r}}=\frac{2s_N(r)}{\abs{\mathbb{S}^{N-1}}r^{\frac{N-2}{2}}}.
\end{equation}
A known fact from probability theory states that the density function for $S_N$, $s_N$, is given by
\begin{equation}\nonumber
s_N(u)=h^{\ast N}(u)
\end{equation}
where $h^{\ast N}$ is the $N-$th iterated convolution of $h$. This completes the proof.
\end{proof}
\begin{remark}\label{rem: properties of h}
It is easy to see that probability density function $h$, associated to the probability density function $f$ as described in the above lemma, is given by
\begin{equation}\label{eq: formula for h}
h(u)=
\begin{cases}
\frac{f\pa{\sqrt{u}}+f\pa{-\sqrt{u}}}{2\sqrt{u}} & u>0 \\
0 & u\leq 0
\end{cases}
\end{equation}
As such, using the convexity of $t\rightarrow t^q$ for any $q>1$, we find that if in addition $f\in L^p(\mathbb{R})$ then
\begin{equation}\label{eq: h is in L^p^prime}
\begin{gathered}
\int h(u)^{p^\prime}du \leq \frac{1}{2} \int_0 ^\infty \frac{f(\sqrt{u})^{p^\prime}+f(-\sqrt{u})^{p^\prime}}{u^{\frac{p^\prime}{2}}}du =\int_\mathbb{R} \frac{f(x)^{p^\prime}}{x^{p^\prime-1}} \\
\leq \int_{[-1,1]} \frac{f(x)^{p^\prime}}{x^{p^\prime-1}}+\int_\mathbb{R} f(x)^{p^\prime}dx \\ 
\leq\left(\int_{[-1,1]} f(x)^{p}dx \right)^{\frac{p^\prime}{p}}
\left(\int_{[-1,1]}\frac{dx}{x^{\frac{p(p^\prime - 1)}{p-p^\prime}}} \right)^{\frac{p-p^\prime}{p^\prime}}
+\int_{f>1} f(x)^{p}dx+\int_{f<1} f(x)dx,
\end{gathered}
\end{equation}
where $p^\prime < p$. Choosing $1<p^\prime<\frac{2p}{1+p}$ we find $h\in L^{p^\prime}\pa{\mathbb{R}}$, showing that $h$ itself gains extra integrability properties in this case. This will serve us later on in Section \ref{sec: chaoticity and entropic chaoticity}.
\end{remark}
\subsection*{Marginals on Kac's Sphere.}
By its definition, chaoticity depends strongly on understanding how finite marginal on Kac's sphere behave. In particular, in our presented cases, we'll be interested to find a simple formula for the $k-$th marginal of probability measures of the form $F_Nd\sigma^N$. To do that we state the following simple lemma, whose proof we'll omit, but can be found in \cite{Einav1}:
\begin{lemma}\label{lem: integration over sphere}
Let $F_N$ be an integrable function on $\mathbb{S}^{N-1}(r)$, then 
\begin{equation}\nonumber
\begin{gathered}
\int_{\mathbb{S}^{N-1}(r)}F_Nd\sigma^N_r=\frac{\abs{\mathbb{S}^{N-j-1}}}{\abs{\mathbb{S}^{N-1}}}\frac{1}{r^{N-2}}\int \pa{r^2-\sum_{i=1}^j v_i^2}_{+}^{\frac{N-j-2}{2}}\\
\pa{\int_{\mathbb{S}^{N-j-1}\pa{\sqrt{r^2-\sum_{i=1}^j v_i^2}}}F_Nd\sigma^{N-j}_{\sqrt{r^2-\sum_{i=1}^j v_i^2}}}dv_1\dots dv_j,
\end{gathered}
\end{equation}
where $g_+ = \max(g,0)$ for a function $g$.
\end{lemma}
Using the above lemma, one can easily show the following:
\begin{lemma}\label{lem: k-th marginal}
Given a distribution function $F_N$ on Kac's sphere, then the probability density function of the $k-$th marginal of the probability measure $F_Nd\sigma^N$ is given by
\begin{equation}\label{eq: k-th marginal}
\begin{gathered}
\Pi_k(F_N)\pa{v_1,\dots,v_k}=\frac{\abs{\mathbb{S}^{N-k-1}}}{\abs{\mathbb{S}^{N-1}}}\frac{1}{N^{\frac{N-2}{2}}}\pa{N-\sum_{i=1}^k v_i^2}_{+}^{\frac{N-k-2}{2}}\\
\pa{\int_{\mathbb{S}^{N-k-1}\pa{\sqrt{r^2-\sum_{i=1}^k v_i^2}}}F_Nd\sigma^{N-k}_{\sqrt{r^2-\sum_{i=1}^j v_k^2}}}.
\end{gathered}
\end{equation}
\end{lemma}
Next we show a simple condition for chaoticity, one we will use later on in Section \ref{sec: chaoticity and entropic chaoticity}:
\begin{lemma}\label{lem: simple chaoticity condition}
Let $\br{F_N}_{N\in\mathbb{N}}$ be a family of distribution functions on Kac's sphere. Assume that there exists a distribution function $f$, on $\mathbb{R}$, such that
\begin{equation}\label{eq: pointwise convergence of marginal}
\lim_{N\rightarrow\infty}\Pi_k(F_N)\pa{v_1,\dots,v_k}=f^{\otimes k}\pa{v_1,\dots,v_k}
\end{equation}
pointwise for all $k\in\mathbb{N}$. Then 
\begin{equation}\label{eq: L^1 convergence of marginal}
\lim_{N\rightarrow\infty}\Norm{\Pi_k(F_N)\pa{v_1,\dots,v_k}-f^{\otimes k}\pa{v_1,\dots,v_k}}_{L^1\pa{\mathbb{R}^k}}=0,
\end{equation}
for all $k\in\mathbb{N}$, and n particular $\br{F_N}_{N\in\mathbb{N}}$ is $f-$chaotic.
\end{lemma}
The proof for this (and a more general statement) can be found in \cite{Einav3}. Since the proof is very simple we will add it here, for completion.
\begin{proof}
Let $k\in\mathbb{N}$ be fixed. Define $g_N=\Pi_k(F_N)+f^{\otimes k}$. By assumption (\ref{eq: pointwise convergence of marginal}) we know that
\begin{equation}\nonumber
\lim_{N\rightarrow\infty}g_N=2f^{\otimes k}=g,
\end{equation}
pointwise and since $\abs{\Pi_k(F_N)-f^{\otimes k}}\leq g_N$, and 
\begin{equation}\nonumber
\int_{\mathbb{R}^k} g_N\pa{v_1,\dots,v_k}dv_1\dots dv_k = \int_{\mathbb{R}^k} g\pa{v_1,\dots,v_k}dv_1\dots dv_k
\end{equation}
for all $N$, we can use the generalised dominated convergence theorem to conclude (\ref{eq: L^1 convergence of marginal}).
\end{proof}
\subsection*{$\alpha$ Stable Processes.}
The bulk of the material presented in this subsection is taken from the excellent book by Feller, \cite{Feller}, as well as the paper \cite{GJT} by Goudon, Junca and Toscani.\\
The concept of stable distribution appears to be very adequate to deal with many real life situations where a strong deviation from the normal central limit theorem is observed. Stable distribution are a generalisation of the normal distribution, and act as attractors for properly scaled and shifted sums of identically distributed variables.\\
One of the simplest way to discuss stable distribution is via their characteristic function. We remind the reader that in the probabilistic context, the characteristic function, $\widehat{\varphi}$, of a probability density $\varphi$ on $\mathbb{R}$ is given by
\begin{equation}\label{eq: char function}
\widehat{\varphi}(\xi)=\int_{\mathbb{R}}e^{ix\xi}\varphi(x)dx.
\end{equation}
\begin{definition}\label{def: stable}
A random variable $U$ is said to be $\alpha-$stable for $0<\alpha<2$, $\alpha\not=1$ if
\begin{equation}\nonumber
\frac{\sum_{i=1}^n X_i}{n^{\frac{1}{\alpha}}}
\end{equation}
has the same probability distribution function as $U$, where $X_i$ are independent copies of $U$. Equivalently, the characteristic function of $U$ is of the form
\begin{equation}\label{eq: char function of alpha stable I}
\widehat{\gamma}_{C_S,\alpha,p,q}(\xi)=e^{-C_S \abs{\xi}^{\alpha}\cdot\frac{\Gamma(3-\alpha)}{\alpha(\alpha-1)}\cos\pa{\frac{\pi\alpha}{2}}\pa{1+i\text{sgn}(\xi)(p-q)\tan{\frac{\alpha\pi}{2}}}},
\end{equation}
with $C_S>0$, $p,q\geq 0$ and $p+q=1$.
\end{definition}
\begin{remark}\label{rem: stricly stable and probability distribution function}
Some books, including Feller's, refer to above definition as \emph{strict stability}. 
\end{remark}
\begin{remark}\label{rem: char funtion for alpha stable II}
Equation (\ref{eq: char function of alpha stable I}) can be rewritten in the form
\begin{equation}\label{eq: char funtion of alpha stable II}
\widehat{\gamma}_{\sigma,\alpha,\beta}(\xi)=e^{-\sigma \abs{\xi}^\alpha \pa{1+i\beta\text{sgn}(\xi)\tan{\frac{\alpha\pi}{2}}}},
\end{equation}
where 
\begin{equation}\nonumber
\begin{gathered}
\sigma=C_S\cdot\frac{\Gamma(3-\alpha)}{\alpha(\alpha-1)}\cos\pa{\frac{\pi\alpha}{2}}>0, \; \beta=p-q.
\end{gathered}
\end{equation}
We will use both forms in accordance to the situation.
\end{remark}
We will now define the Domain of Attraction of a stable distribution (which we will identify via its characteristic function), as well as the Natural Domain of Attraction and the Fourier Domain of Attraction. 
\begin{definition}\label{def: DA}
The \emph{Domain of Attraction} (in short, DA) of $\widehat{\gamma}_{\sigma,\alpha,\beta}$ is the set of all real random variables $X$ such that there exist sequences $\br{a_n}_{n\in\mathbb{N}}>0$ and $\br{b_n}_{n\in\mathbb{N}}\in\mathbb{R}$ such that
\begin{equation}\label{eq: DA weak equation}
\frac{\sum_{i=1}^n X_i}{a_n}-nb_n \underset{n\rightarrow\infty}{\longrightarrow} U,
\end{equation}
where $X_i$ are independent copies of $X$, $U$ is the real random variable with characteristic function $\widehat{\gamma}_{\sigma,\alpha,\beta}$ and the limit is to be understood in the weak sense. Equivalently, one can prove that the DA of $\widehat{\gamma}_{\sigma,\alpha,\beta}$ is the set of all real random variables $X$, whose characteristic function $\widehat{\psi}$ satisfies
\begin{equation}\label{eq: char charactirization of DA}
n\left(\widehat{\psi} \left(\frac{\xi}{a_n} \right)e^{-ib_n \xi}-1 \right) \underset{n\rightarrow\infty}{\longrightarrow} -\sigma|\xi|^\alpha \left(1+i\beta\text{sgn}(\xi)\tan\left(\frac{\pi\alpha}{2} \right) \right),
\end{equation}

where $\br{a_n}_{n\in\mathbb{N}}$ and $\br{b_n}_{n\mathbb{N}}$ are sequences as in (\ref{eq: DA weak equation}) (See \cite{Feller}).
\end{definition}
\begin{definition}\label{def: NDA}
The \emph{Natural Domain of Attraction} (in short, NDA) of $\widehat{\gamma}_{\sigma,\alpha,\beta}$ is the subset of the DA of $\widehat{\gamma}_{\sigma,\alpha,\beta}$  for which $a_n=n^{\frac{1}{\alpha}}$ and $b_n=0$ are applicable as a sequences in (\ref{eq: DA weak equation}).
\end{definition}
\begin{definition}\label{def: FDA}
The \emph{Fourier Domain of Attraction} (in short, FDA) of $\widehat{\gamma}_{\sigma,\alpha,\beta}$  is the set of all real random variables $X$ whose characteristic function $\widehat{\psi}$ satisfies
\begin{equation}\label{eq: char charactirization of FDA}
\widehat{\psi}(\xi)=1 -\sigma\abs{\xi}^\alpha \left(1+i\beta\text{sgn}(\xi)\tan\left(\frac{\pi\alpha}{2} \right) \right)+\eta(\xi),
\end{equation}
where $\frac{\eta(\xi)}{|\xi|^\alpha} \in L^\infty$ and $\frac{\eta(\xi)}{|\xi|^\alpha}\underset{\xi\rightarrow 0}{\longrightarrow}0$. The function $\eta$ is called \emph{the reminder function} of $\widehat{\psi}$.
\end{definition}
The next theorem, taken from \cite{GJT}, is important for our local central limit theorem. The fact that it only works in $\mathbb{R}$ will affect the lower semi-continuity property, discussed in Section \ref{sec: lsc and stability}.
\begin{theorem}\label{thm: FDA and NDA}
For any $\widehat{\gamma}_{\sigma,\alpha,\beta}$ we have that the NDA equals the FDA.
\end{theorem}
Due to its importance, we will present a full proof for this theorem. The proof relies on the following technical lemma (again, taken from \cite{GJT}):
\begin{lemma}\label{lem: technical FDA NDA lemma}
Let $g:\mathbb{R}\setminus \left\lbrace 0 \right\rbrace \rightarrow \mathbb{R}$ be a continuous function that satisfies $\lim_{n\rightarrow\infty}g \left(\frac{x}{n} \right)=0$ for any $x\in\mathbb{R} \setminus \left\lbrace 0 \right\rbrace$. Then $\lim_{x\rightarrow 0 }g(x)=0$.
\end{lemma}
We leave the proof to the Appendix, and show how one can prove Theorem \ref{thm: FDA and NDA} using it.
\begin{proof}[Proof of Theorem \ref{thm: FDA and NDA}]
We start with the easy direction. Assume that $\widehat{\psi}$ is in the FDA of $\widehat{\gamma}_{\sigma,\alpha,\beta}$. We have that
\[n \left( \widehat{\psi} \left( \frac{\xi}{n^{\frac{1}{\alpha}}} \right) -1 \right) =- n\cdot \frac{\sigma|\xi|^\alpha}{n} \left(1+i\beta\text{sgn}\left( \frac{\xi}{n^{\frac{1}{\alpha}}}\right)\tan\left(\frac{\pi\alpha}{2} \right) \right)+n\eta\left( \frac{\xi}{n^{\frac{1}{\alpha}}}\right).\]
\[=-\sigma \abs{\xi}^\alpha \pa{1+i\beta\text{sgn}(\xi)\tan\pa{\frac{\pi\alpha}{2}}}+\abs{\xi}^\alpha \cdot \frac{\eta\pa{\frac{\xi}{n^{\frac{1}{\alpha}}}}}{\pa{\frac{\xi}{n^{\frac{1}{\alpha}}}}^\alpha},\]
concluding the desired result.\\
Conversely, assume that $\widehat{\psi}$ is in the NDA of $\widehat{\gamma}_{\sigma,\alpha,\beta}$ and define 
\[\eta(\xi)=\widehat{\psi}(\xi)-1+\sigma |\xi|^\alpha \left(1+i\beta\text{sgn}(\xi)\tan\left(\frac{\pi\alpha}{2} \right) \right).\]
We have that for any $\xi\not=0$
\[\frac{\eta \left(\frac{\xi}{n^{\frac{1}{\alpha}}} \right)}{\left\lvert \frac{\xi}{n^{\frac{1}{\alpha}}}\right\rvert^{\alpha}}
=\frac{1}{|\xi|^\alpha}\left(n\left(\widehat{\psi}\left(\frac{\xi}{n^{\frac{1}{\alpha}}}\right)-1 \right) +\sigma |\xi|^\alpha \left(1+i\beta\text{sgn}(\xi)\tan\left(\frac{\pi\alpha}{2} \right) \right)\right)\underset{n\rightarrow\infty}{\longrightarrow}0.\]
Defining $g(\xi)=\frac{\eta(\xi)}{|\xi|^\alpha}$ we find that $g$ is continuous on $\mathbb{R}\setminus \left\lbrace 0 \right\rbrace$ and 
\[g\left( \frac{\xi}{n^{\frac{1}{\alpha}}}\right)\underset{n\rightarrow\infty}{\longrightarrow}0\]
for any $\xi\not=0$. A simple modification of Lemma \ref{lem: technical FDA NDA lemma} proves that $\lim_{\xi\rightarrow 0}g(\xi)=0$. This also shows, since $\eta$ is continuous, that $\frac{\eta(\xi)}{|\xi|^{\alpha}}$ is bounded around $\xi=0$. For $|\xi|>\delta$ we have that
\[\frac{|\eta(\xi)|}{|\xi|^\alpha} \leq \frac{2}{\delta^\alpha}+\sigma\left(1+\abs{\beta}\left\lvert\tan\left(\frac{\pi\alpha}{2} \right)\right\rvert \right), \]
proving that $\frac{\eta(\xi)}{|\xi|^{\alpha}}\in L^\infty$, and the result follows.
\end{proof}
Theorem \ref{thm: FDA and NDA} gives us a very convenient approximation for the characteristic function of any real random variable in the NDA of $\widehat{\gamma}_{\sigma,\beta,\alpha}$, one we will use quite strongly in the next section. For now, we finish by quoting a theorem from Feller's book, \cite{Feller}, giving conditions for a real random variable to be in the NDA of a stable distribution.
\begin{theorem}\label{thm: NDA conditions from Feller}
Let $F$ be a probability distribution function of a real random variable, $X$, that has zero mean, and let $1<\alpha< 2$. Denote by
\begin{equation}\label{eq: definition of mu from DA feller theorem}
\mu(x) = \int_{-x}^x y^2 F(dy).
\end{equation}
If
\begin{enumerate}[(i)]
\item \begin{equation}\label{eq: conditions on mu from feller for DA}
\mu(x)\underset{x\longrightarrow\infty}{\sim} x^{2-\alpha} L(x),
\end{equation} 
where $L$ is slowly varying (i.e. $\frac{L(tx)}{L(x)}\underset{x\rightarrow\infty}{\longrightarrow}1$ for any $t>0$).
\item
\begin{equation}\label{eq: conditions on F to find p and q}
\begin{gathered}
\frac{1-F(x)}{1-F(x)+F(-x)}\underset{x\rightarrow\infty}{\longrightarrow}p, \\
\frac{F(-x)}{1-F(x)+F(-x)}\underset{x\rightarrow\infty}{\longrightarrow}q.
\end{gathered}
\end{equation}
\item There exists a sequence $\br{a_n}_{n\in\mathbb{N}}>0$ such that
\begin{equation}\label{eq: conditions on a_n for Feller DA}
\frac{n\mu(a_n)}{a_n^2} \underset{n\rightarrow\infty}{\longrightarrow}C_{S}.
\end{equation}
\end{enumerate}
Then $X$ is in the DA of $\widehat{\gamma}_{C_S,\alpha,p,q}$ with $\br{a_n}_{n\in\mathbb{N}}$ found in $(iii)$ and $b_n=0$. 
\end{theorem}
\begin{remark}\label{rem:alpha less than 1}
It is worth mentioning that a similar, less restrictive theorem, holds in the case $0<\alpha<1$. Since we will not use it in this work, we decided to exclude it from this section. For more information we refer the interested reader to \cite{Feller}.
\end{remark}
\begin{remark}\label{rem: our particular case with respect to feller's thm}
Of particular interest to us are the following cases:
\begin{itemize}
\item if in condition $(i)$ of Theorem \ref{thm: NDA conditions from Feller} one has that $L(x)\underset{x\rightarrow\infty}{\sim}C_S$ then the sequence 
\[a_n=n^{\frac{1}{\alpha}}\]
will be suitable for condition $(iii)$ of the same theorem.
\item If the probability distribution function, $F(x)$, is supported in $[\kappa,\infty)$ for some $\kappa\in \mathbb{R}$ then condition $(ii)$ of Theorem \ref{thm: NDA conditions from Feller}
 is immediately satisfied with $p=1$ and $q=0$.
 \end{itemize}
\end{remark}
We are now ready to begin with the main technical tool of this paper - a local L\'evy central limit theorem.

\section{L\'evy Type Local Central Limit theorem.}\label{sec: local central limit}
The central limit theorem is one of those rare theorems that is of immense importance both theoretically and in practice. The first version to be discovered involved convergence to a normal distribution of certain rescaled and shifted sums of independent identically distributed real variables, but as more and more cases of deviation from such nice distribution were observed, a more general version of a central limit theorem, one involving the stable distribution, was investigated. Of particular interest in our field of study is the concept of a \emph{local central limit theorem}, that is - a central limit theorem that doesn't only apply to the probability distribution function but to the probability density function as well.\\
In this section we will present such theorem, extending results obtained in \cite{CCRLV} for the case where one has a bounded fourth moment. The proofs associated with the local limit theorem are modelled on similar ideas to those presented in the above paper, but there are some significant changes, on which we will remark.\\
The main idea of the proof is to evaluate the supremum of the difference between the probability density functions using inversion formula and their characteristic functions. An integral will emerge, one we will have to divide into two domains: low and high frequencies. The domain of low frequencies will be taken cared of by requiring that the characteristic function would be in the NDA of some stable distribution. The high frequency domain is what we'll deal with presently.
\begin{theorem}\label{thm:high frequency}
Let $g$ be a probability density function on $\mathbb{R}$ such that 
\begin{equation}\label{eq: any moment bound on g}
E_\lambda=\int_{\mathbb{R}}|x|^\lambda g(x)dx<\infty,
\end{equation} 
for some $\lambda>0$, and \begin{equation}\label{eq: entropy bound on g}
H(g)=\int_{\mathbb{R}}g(x) \log g(x) dx<\infty.
\end{equation} 
Then for any $\beta>0$, there exists $\eta=\eta\left(\beta, H(g),E_\lambda \right)>0$ such that if $|\xi|>\beta$ then $|\widehat{g}(\xi)|\leq 1- \eta$. Moreover, given $\tau>0$ one can get the estimation 
\begin{equation}\label{eq: estimation of widehat(g) in the interval}
|\widehat{g}(\xi)| \leq 1- \beta^{2+\tau} + \phi_{\tau}(\beta),
\end{equation}
for $\beta<\beta_0$ small enough, where $\frac{\phi_\delta(\tau)}{\beta^{2+\tau}}\underset{\beta\rightarrow 0 }{\longrightarrow}0$. 
\end{theorem}
\begin{remark}\label{rem: high frequencies changes}
The proof of the first part of the above theorem, to be presented shortly, is very similar to the proof found in \cite{CCRLV}. The novelty of our approach manifests itself mainly in (\ref{eq: estimation of widehat(g) in the interval}), where an explicit distance from $1$ is given. The surprising part is that to show this estimation no new machinery is required, only an intermediate approximation. 
\end{remark}
\begin{proof}
For a given $\xi\in \mathbb{R}$ we can find a $z\in\mathbb{R}$ such that 
\[|\widehat{g}(\xi)|=\widehat{g}(\xi)e^{-2\pi i \xi z}.\]
By the definition of the Fourier transform, and the fact that $\widehat{g}(0)=1$, we have that
\[|\widehat{g}(\xi)|=\int_{\mathbb{R}}g(x)e^{-2\pi i (x+z) \xi}dx=1-\int_{\mathbb{R}}g(x)\left(1-e^{-2\pi i (x+z) \xi}\right)dx.\]
Since $|\widehat{g}|$ is real we find that
\begin{equation}\label{eq: intital estimation of g-hat}
\begin{gathered}
|\widehat{g}(\xi)|=1-\int_{\mathbb{R}}g(x)\left(1-\cos\left(2\pi (x+z) \xi\right)\right)dx \\
\leq 1-\int_{B}g(x)\left(1-\cos\left(2\pi (x+z) \xi\right)\right)dx
\end{gathered}
\end{equation}
for any measurable set $B$.\\
Define: 
\[B_{\delta,R}=\left\lbrace x\in[-R,R] \; | \; 1-\cos\left(2\pi(z+x)\xi\right)\leq \delta  \right\rbrace,\]
where $\delta$ and $R$ are to be specified later. From its definition, and (\ref{eq: intital estimation of g-hat}), we conclude that
\begin{equation}\label{eq: second estimation of g-hat}
\begin{gathered}
|\widehat{g}(\xi)|\leq 1-\int_{[-R,R]\setminus B_{\delta,R}}g(x)\left(1-\cos\left(2\pi (x+z) \xi\right)\right)dx 
\\
\leq 1-\delta\int_{[-R,R]\setminus B_{\delta,R}}g(x)dx.  
\end{gathered}
\end{equation}
Next we notice that $x\in  B_{\delta,R}$ if and only if $x\in[-R,R]$ and 
\[|2\pi(z+x)\xi+2\pi k|\leq \arccos(1-\delta)\]
for some $k\in\mathbb{Z}$. Since $\arccos(1-\delta)\leq \sqrt{2\delta}$ 
we conclude that if $x\in B_{\delta,R}$ then, for some $k\in\mathbb{Z}$, 
\begin{equation}\label{eq: interval estimation}
\left\lvert x-\left( \frac{k}{\xi}-z \right) \right\rvert \leq \frac{\sqrt{2\delta}}{2\pi|\xi|}.
\end{equation}
We denote by $I_k$ the closed intervals centred in $\frac{k}{\xi}-z$, with radius $\frac{\sqrt{2\delta}}{2\pi\abs{\xi}}$. Since the distance between the centres of any two $I_k-$s is at least $\frac{1}{|\xi|}$, while the length of each interval is at most $\frac{1}{\pi|\xi|}$, if we pick $\delta<\frac{1}{2}$, we conclude that the intervals $I_k-$s are mutually disjoint.\\ 
From (\ref{eq: interval estimation}) we see that the set $B_{\delta,R}$ is contained in a union of $I_k-$s.\\
Let $n$ be the number of $k\in\mathbb{Z}$ such that $\frac{k}{\xi}-z\in[-R,R]$. All such $k-$s, but possibly the biggest and smallest $k$, satisfy that $I_k \subset [-R,R]$. Thus,
\[(n-2)\cdot \frac{1}{\pi|\xi|} \leq \sum_{I_k \subset [-R,R]} |I_k|\leq 2R.\]
With $\abs{\cdot}$ denoting the Lebesgue measure, we conclude that
\begin{equation}\label{eq: lebesgue measure of B}
\begin{gathered}
\abs{ B_{\delta,R}}\leq n\cdot \frac{\sqrt{2\delta}}{\pi|\xi|}\leq \left(2R+\frac{2}{\pi |\xi|}\right)\cdot \sqrt{2\delta}\leq 2R\left(1+\frac{1}{R\beta}\right)\cdot \sqrt{2\delta}.
\end{gathered}
\end{equation}
At this point we will use the entropy and moment condition on $g$ to connect between the known value $\abs{B_{\delta,R}}$ and the desired value $\int_{B_{\delta,R}}g(x)dx$. To do that we will use the relative entropy (see Definition \ref{def: relative entropy }) and the following known inequality: 
\begin{equation}\label{eq: relative entropy inequality}
\mu(B)\leq \frac{2H(\mu | \nu)}{\log\left(1+\frac{H(\mu | \nu)}{\nu(B)} \right)},  
\end{equation}
where $\mu$ and $\nu$ are regular probability measure on $\mathbb{R}$ and $B$ is a measurable set.\\
Define 
\begin{equation}\label{eq: definition of the measures for relative entropy}
d\mu(x)=\frac{\chi_{[-R,R]}(x)g(x)}{\int_{[-R,R]} g(x)dx}dx, \quad d\nu(x)=\frac{\chi_{[-R,R]}(x)}{2R}dx.
\end{equation} 
We have that $\frac{d\mu}{d\nu}(x)=\frac{2R \chi_{[-R,R]}(x)g(x)}{\int_{[-R,R]} g(x)dx}$ and
\begin{equation}\label{eq: high freq estimation I}
\begin{gathered}
H(\mu | \nu)=\int_{[-R,R]} \log\left( \frac{2R g(x)}{\int_{[-R,R]}g(x)dx} \right) \frac{g(x)}{\int_{[-R,R]}g(x)dx}dx\\
=\log(2R)-\log \left(\int_{[-R,R]}g(x)dx \right)+\frac{1}{\int_{[-R,R]}g(x)dx}\int_{[-R,R]}g(x)\log g(x) dx\\
\leq \log(2R)-\log \left(1- \frac{E_\lambda}{R^\lambda} \right)+\frac{1}{1- \frac{E_\lambda}{R^\lambda}}\int g(x)|\log g(x)| dx.
\end{gathered}
\end{equation}
We have used the fact that
\begin{equation}\label{eq: high freq estimation II}
\int_{[-R,R]}g(x)dx=1-\int_{|x|>R}g(x)dx \geq 1-\frac{1}{R^\lambda}\int_{|x|>R}|x|^\lambda g(x)dx \geq 1- \frac{E_\lambda}{R^\lambda}.
\end{equation}
We will now turn our attention to the term $\int g(x)\abs{\log (g(x))}dx$. For \textit{any} positive function $\psi(x)$, we have that 
\begin{equation}\nonumber
\psi(x)\left(\frac{g(x)}{\psi(x)}\log \left(\frac{g(x)}{\psi(x)} \right) -\frac{g(x)}{\psi(x)}+1\right)\geq 0.
\end{equation} 
Thus, for any measurable set $A$ we have that
\[\int_A g(x)\log g(x)dx \geq \int_A g(x)\log \psi(x)dx+\int_A g(x) - \int_A \psi(x)dx,\]
when the right hand side is finite. Choosing  $\psi(x)=e^{-|x|^\lambda}$ and $A=\left\lbrace g<1 \right\rbrace$ we find that 
\begin{equation}\label{eq: entropy bound on g<1}
\begin{gathered}
\left\lvert \int_{g<1} g(x)\log g(x)dx \right\rvert=- \int_{g<1} g(x)\log g(x) \\
\leq \int_{g<1} |x|^\lambda g(x)dx-\int_{g<1}g(x)dx + \int_{g<1}\psi(x)dx < E_\lambda+ C_\lambda. 
\end{gathered}
\end{equation}
where $C_\lambda=\int \psi(x)dx$. Since
\begin{equation}\nonumber
\int g(x)|\log(g(x)|=H(g)-2\int_{g<1}g(x)\log g(x)dx.
\end{equation} 
we conclude that
\begin{equation}\label{eq: relative entropy bound}
H(\mu | \nu) \leq \log(2R)-\log \left(1- \frac{E_\lambda}{R^\lambda} \right)+\frac{H(g)+2E_\lambda+2C_\lambda}{1- \frac{E_\lambda}{R^\lambda}}.
\end{equation}
Together with (\ref{eq: lebesgue measure of B}) and (\ref{eq: relative entropy inequality}) we find that
\begin{equation}\label{eq: mu(B_delta,R)}
\mu(B_{\delta,R})\leq \frac{2\log(2R)-2\log \left(1- \frac{E_\lambda}{R^\lambda} \right)+\frac{2H(g)+4E_\lambda+4C_\lambda}{1- \frac{E_\lambda}{R^\lambda}}}
{\log\left(1+\frac{\log(2R)-\log \left(1- \frac{E_\lambda}{R^\lambda} \right)+\frac{H(g)+2E_\lambda+2C_\lambda}{1- \frac{E_\lambda}{R^\lambda}}}{2R\left( 1+\frac{1}{R\beta} \right)\sqrt{2\delta}} \right)}.
\end{equation}
Next, we notice that
\[\int_{[-R,R]\setminus B_{\delta,R}} g(x)dx=\left(\int_{[-R,R]} g(x)dx\right)\mu\left([-R,R] \setminus B_{\delta,R} \right) \geq \left(1-\frac{E_\lambda}{R^\lambda} \right)\left(1-\mu(B_{\delta,R}) \right)\]
which, along with (\ref{eq: second estimation of g-hat}) and (\ref{eq: mu(B_delta,R)}) gives us the following control:
\begin{equation}\label{eq: final result unpolished}
\begin{gathered}
|\widehat{g}(\xi)|\leq 1-\delta\cdot \left(1-\frac{E_\lambda}{R^\lambda} \right)\left(1-\frac{2\log(2R)-2\log \left(1- \frac{E_\lambda}{R^\lambda} \right)+\frac{2H(g)+4E_\lambda+4C_\lambda}{1- \frac{E_\lambda}{R^\lambda}}}
{\log\left(1+\frac{\log(2R)-\log \left(1- \frac{E_\lambda}{R^\lambda} \right)+\frac{H(g)+2E_\lambda+2C_\lambda}{1- \frac{E_\lambda}{R^\lambda}}}{2R\left( 1+\frac{1}{R\beta} \right)\sqrt{2\delta}} \right)} \right)
\end{gathered}
\end{equation}
At this point we can choose $R$ and $\delta<\frac{1}{2}$ appropriately. For any $\tau>0$ we choose $\delta=\beta^{2+\tau}$ and $R=-\log \beta$ we find that for $\beta$ going to zero
\[\frac{2\log(2R)-2\log \left(1- \frac{E_\lambda}{R^\lambda} \right)+\frac{2H(g)+4E_\lambda+4C_\lambda}{1- \frac{E_\lambda}{R^\lambda}}}
{\log\left(1+\frac{\log(2R)-\log \left(1- \frac{E_\lambda}{R^\lambda} \right)+\frac{H(g)+2E_\lambda+2C_\lambda}{1- \frac{E_\lambda}{R^\lambda}}}{2R\left( 1+\frac{1}{R\beta} \right)\sqrt{2\delta}} \right)} \approx \frac{2\log (-\log (\beta))}{\log \left(1+\frac{\log(-\log(\beta))}{-2\sqrt{2}\beta^{1+\frac{\tau}{2}}\log(\beta)+2\sqrt{2}\beta^{\frac{\tau}{2}}}\right)}\]
\[\approx \frac{2\log (-\log (\beta))}{\log(\log(-\log (\beta)))-\frac{\tau}{2}\cdot\log (\beta)}\underset{\beta\rightarrow 0}{\longrightarrow}0.\]
Thus, 
\[|\widehat{g}(\xi)| \leq 1- \beta^{2+\tau} + \phi_\tau(\beta),\]
where $\frac{\phi_\tau(\beta)}{\beta^{2+\tau}}\underset{\beta\rightarrow 0 }{\longrightarrow}0$. 
\end{proof}
Before we state and prove our main L\'evy central limit theorem, we state a simple technical lemma, one that will be proven in the appendix. A similar argument can be found in \cite{GJT}.   
\begin{lemma}\label{lem: g in NDA implies g is exponential at low frequencies}
Let $\widehat{g}$ be the characteristic function of a random real variable $X$ that is in the NDA of $\widehat{\gamma}_{\sigma,\alpha,\beta}$. Then there exists $\beta_0>0$ such that for all $|\xi|<\beta_0$ we have that
\begin{equation}\label{eq: g in NDA implies g is exponential at low frequencies}
\abs {\widehat{g}(\xi)} \leq e^{-\frac{\sigma\abs{\xi}^{\alpha}}{2}}.
\end{equation}
\end{lemma}

\begin{theorem}\label{thm: main approximation themroem}
Let $g$ be the probability density function of a random real variable $X$. Assume that $g\in L^{p}(\mathbb{R})$ for some $p>1$ and $g$ is in the NDA of $\gamma_{\sigma,\alpha,\beta}$ for some $\sigma>0$, $\beta$ and $1<\alpha<2$. Assume in addition that $g$ has finite moment of some order. Define 
\begin{equation}\nonumber
g_N(x)=N^{\frac{1}{\alpha}}g^{\ast N}\left(N^{\frac{1}{\alpha}}x \right),
\end{equation}
and 
\begin{equation}\label{eq: def of gamma_sigma}
\gamma_{\sigma,\alpha,\beta}(x)=\frac{1}{2\pi}\int_{\mathbb{R}}\widehat{\gamma}_{\sigma,\alpha,\beta}(\xi)e^{i \xi x}d\xi.
\end{equation}
Then, for any positive sequence $\br{\beta_N}_{N\rightarrow\infty}$ that converges to zero as $N$ goes to infinity, any $\tau>0$ and $N$ large enough we have that
\begin{equation}\label{eq: main approxiamtion equation}
\begin{gathered}
\Norm {g_N-\gamma_{\sigma,\alpha,\beta}}_\infty \leq C_{g,\alpha}\Bigg(N^{\frac{1}{\alpha}}(1-\beta^{2+\tau}_N+\phi_\tau(\beta_N))^{N-q} 
+e^{-\frac{\sigma N\beta_N^\alpha}{2}} \\
+ \omega_{\eta}(\beta_N)+2\sigma \beta_N^\alpha \pa{1+\beta^2 \tan^2\pa{\frac{\pi\alpha}{2}}}\Bigg)=\epsilon_{\tau}(N),
\end{gathered}
\end{equation}
where 
\begin{enumerate}[(i)]
\item $C_{g,\alpha}>0$ is a constant depending only on $g$, its moments and $\alpha$.
\item  $q$ can be chosen to be the H\"older conjugate of $\min(2,p)$.
\item  $\phi_{\tau}$ satisfies
\begin{equation}\nonumber
\lim_{x\rightarrow 0}\frac{\phi_{\tau}(x)}{|x|^{2+\tau}}=0,
\end{equation}
\item $\eta$ is the reminder function of $\widehat{g}$, defined in Definition \ref{def: FDA}, and $\omega_\eta(\beta)=\sup_{|x|\leq \beta}\frac{\abs{\eta(x)}}{\abs{x}^\alpha}$.
\end{enumerate}
\end{theorem}
The proof of Theorem \ref{thm: main approximation themroem} is similar in nature to proofs presented in \cite{CCRLV,GJT}, yet there are some differences. The main one is the explicit estimation, per $N$, of the distance between $g_N(x)$ and $\gamma_{\sigma,\alpha,\beta}$. 
\begin{proof}
We start by noticing that 
\begin{equation}\nonumber
\widehat{g_N}(\xi)=\widehat{g}^N \pa{\frac{\xi}{N^{\frac{1}{\alpha}}}},
\end{equation}
and from the inversion formula for characteristic functions (see \cite{Feller}) we have that $\widehat{\gamma}_{\sigma,\alpha,\beta}$ is the characteristic function of $\gamma_{\sigma,\alpha,\beta}$.\\
Since $g\in L^1 (\mathbb{R})\cap L^{p}(\mathbb{R})$ we conclude that $g\in L^{p^\prime}(\mathbb{R})$ for any $1\leq p^{\prime}\leq p$. Thus, its characteristic function belongs to some $L^q(\mathbb{R})$ for some $q>1$. One can choose $q$ to be the H\"older conjugate of $\min(2,p)$. For any $N>q$ we have that
\begin{equation}\nonumber
\int_{\mathbb{R}}\abs{\widehat{g_N}(\xi)}d\xi \leq \Norm{\widehat{g}}^{N-q}_\infty \int_{\mathbb{R}}\abs{\widehat{g}\pa{\frac{\xi}{N^{\frac{1}{\alpha}}}}}^q d\xi \leq N^{\frac{1}{\alpha}}\Norm{\widehat{g}}^q_{L^q}<\infty.
\end{equation}
This implies that we can use the inversion formula for $g$, and as such, for any $x\in\mathbb{R}$:
\begin{equation}\label{eq: approximation heorem, low frequency I}
\begin{gathered}
\abs{g_N(x)-\gamma_{\sigma,\alpha,\beta}(x)} \leq \frac{1}{2\pi}\int_{\mathbb{R}}\abs{\widehat{g}^N \pa{\frac{\xi}{N^{\frac{1}{\alpha}}}}-\widehat{\gamma}_{\sigma,\alpha,\beta}(\xi)}d\xi \\
=\frac{1}{2\pi}\int_{\mathbb{R}}\abs{\widehat{g}^N \pa{\frac{\xi}{N^{\frac{1}{\alpha}}}}-\widehat{\gamma}^N_{\sigma,\alpha,\beta}\pa{\frac{\xi}{N^{\frac{1}{\alpha}}}}}d\xi \\
\leq \frac{1}{2\pi}\int_{\abs{\xi}<\beta_N N^{\frac{1}{\alpha}}}\abs{\widehat{g}^N \pa{\frac{\xi}{N^{\frac{1}{\alpha}}}}-\widehat{\gamma}^N_{\sigma,\alpha,\beta}\pa{\frac{\xi}{N^{\frac{1}{\alpha}}}}}d\xi \\
\frac{1}{2\pi}\int_{\abs{\xi}>\beta_N N^{\frac{1}{\alpha}}}\abs{\widehat{g}^N \pa{\frac{\xi}{N^{\frac{1}{\alpha}}}}}d\xi+\frac{1}{2\pi}\int_{\abs{\xi}>\beta_N N^{\frac{1}{\alpha}}}\abs{\widehat{\gamma}_{\sigma,\alpha,\beta}(\xi)}d\xi \\
=I_1+I_2+I_3.
\end{gathered}
\end{equation}
The partition in (\ref{eq: approximation heorem, low frequency I}) corresponds to the low-high frequencies domains we referred to at the beginning of the section. We wil start with estimating $I_1$.\\
Since $\widehat{g}$ is in the NDA of $\widehat{\gamma}_{\sigma,\alpha,\beta}$, Theorem \ref{thm: FDA and NDA} assures us that $\widehat{g}$ is in the FDA of $\widehat{\gamma}_{\sigma,\alpha,\beta}$ and there exists a reminder function, $\eta$, such that
\begin{equation}\label{eq: approximation heorem, low frequency II}
\abs{\widehat{g}(\xi)-\widehat{\gamma}_{\sigma,\alpha,\beta}(\xi)}=\abs{\eta(\xi)}+\abs{\eta_\gamma(\xi)},
\end{equation}
with 
\begin{equation}\label{eq: approximation heorem, low frequency III}
\abs{\eta_\gamma(\xi)} \leq 2 \sigma^2 \abs{\xi}^{2\alpha}\pa{1+\beta^2 \tan^2\pa{\frac{\pi\alpha}{2}}}
\end{equation}
when $\abs{\xi}<\beta_1$ for some small $\beta_1>0$. Thus,
\begin{equation}\label{eq: approximation heorem, low frequency IV}
\sup_{|\zeta|<\beta_N}\frac{\abs{\widehat{g}(\zeta)-\widehat{\gamma}_{\sigma,\alpha,\beta}(\zeta)}}{\abs{\zeta}^\alpha} \leq \omega_{\eta}(\beta_N) + 2\sigma \beta_N^\alpha \pa{1+\beta^2 \tan^2\pa{\frac{\pi\alpha}{2}}}
\end{equation}
for $N$ large enough such that $\beta_N<\beta_1$.\\
Next, we see that
\begin{equation}\label{eq: approximation heorem, low frequency V}
\begin{gathered}
\abs{\widehat{g}^N \pa{\frac{\xi}{N^{\frac{1}{\alpha}}}}-\widehat{\gamma}^N_{\sigma,\alpha,\beta}\pa{\frac{\xi}{N^{\frac{1}{\alpha}}}}} \\
\leq  \abs{\widehat{g} \pa{\frac{\xi}{N^{\frac{1}{\alpha}}}}-\widehat{\gamma}_{\sigma,\alpha,\beta}\pa{\frac{\xi}{N^{\frac{1}{\alpha}}}}}\sum_{k=0}^{N-1}\abs{\widehat{g}\pa{\frac{\xi}{N^{\frac{1}{\alpha}}}}}^k \abs{\widehat{\gamma}_{\sigma,\alpha,\beta}\pa{\frac{\xi}{N^{\frac{1}{\alpha}}}}}^{N-1-k}.
\end{gathered}
\end{equation}
Picking $N$ such that $\frac{\abs{\xi}}{N^{\frac{1}{\alpha}}}<\beta_N<\beta_0$ from Lemma \ref{lem: g in NDA implies g is exponential at low frequencies} we find that
\begin{equation}\label{eq: approximation heorem, low frequency VI}
\begin{gathered}
\sum_{k=0}^{N-1}\abs{\widehat{g}\pa{\frac{\xi}{N^{\frac{1}{\alpha}}}}}^k \abs{\widehat{\gamma}_{\sigma,\alpha,\beta}\pa{\frac{\xi}{N^{\frac{1}{\alpha}}}}}^{N-1-k}
\leq \sum_{k=0}^{N-1}e^{-\frac{\sigma k \abs{\xi}^\alpha}{2N}}\cdot e^{-\frac{\sigma (N-k-1)\abs{\xi}^\alpha}{N} } \\
\leq N e^{-\frac{\sigma (N-1)\abs{\xi}^\alpha}{2N}} \leq N e^{-\frac{\sigma \abs{\xi}^\alpha}{4}},
\end{gathered}
\end{equation}
when $N\geq 2$. Combining (\ref{eq: approximation heorem, low frequency IV}), (\ref{eq: approximation heorem, low frequency V}) and (\ref{eq: approximation heorem, low frequency VI}) we see that
\begin{equation}\label{eq: estimation of I_1}
\begin{gathered}
I_1 \leq \frac{\omega_{\eta}(\beta_N) + 2\sigma \beta_N^\alpha \pa{1+\beta^2 \tan^2\pa{\frac{\pi\alpha}{2}}}}{2\pi}\int_{\abs{\xi} < \beta_N N^{\frac{1}{\alpha}}}\frac{\abs{\xi}^\alpha}{N}\cdot Ne^{-\frac{\sigma \abs{\xi}^\alpha}{4}}d\xi\\
\leq C\pa{\omega_{\eta}(\beta_N) + 2\sigma \beta_N^\alpha \pa{1+\beta^2 \tan^2\pa{\frac{\pi\alpha}{2}}}},
\end{gathered}
\end{equation}
where $C=\int_{\mathbb{R}}\abs{\xi}^\alpha e^{-\frac{\sigma \abs{\xi}^\alpha}{4}}d\xi$. Next, we estimate $I_2$.\\
The expression $I_2$ is connected to the high frequency theorem, Theorem \ref{thm:high frequency}, and as such we need to check that its conditions are satisfied. From the conditions given in the statement of our theorem, we know that there exists $\lambda>0$ such that $E_\lambda<\infty$, using the notations of Theorem \ref{thm:high frequency}. We only need to show that $H(g)<\infty$. Indeed, since $g\in L^p(\mathbb{R})$ for some $p>1$ we have that
\begin{equation}\nonumber
\int_{\mathbb{R}}g(x)\abs{\log g(x)}dx = -\int_{g<1}g(x)\log g(x) dx + \int_{g\geq 1}g(x)\log g(x)dx.
\end{equation}
We already showed in the proof of Theorem \ref{thm:high frequency} that $-\int_{g<1}g(x)\log g(x) dx<\infty$, and since we can always find $C_p>0$ such that $\log x \leq C_p x^{p-1}$ for $x\geq 1$ we conclude that
\begin{equation}\nonumber
\int_{g\geq 1}g(x)\log g(x)dx \leq C_p \Norm{g}^p_{L^p(\mathbb{R})}<\infty,
\end{equation}
showing that $H(g)<\infty$. Thus, for any $\tau>0$ and for $\beta$ small enough we have that
\begin{equation}\nonumber
|\widehat{g}(\xi)| \leq 1- \beta^{2+\tau} + \phi_{\tau}(\beta),
\end{equation}
with $\frac{\phi_\delta(\tau)}{\beta^{2+\tau}}\underset{\beta\rightarrow 0 }{\longrightarrow}0$. \\
Using the above, we conclude that
\begin{equation}\label{eq: estimation of I_2}
\begin{gathered}
I_2 = \frac{N^{\frac{1}{\alpha}}}{2\pi}\int_{\abs{\xi}>\beta_N}\abs{\widehat{g}(\xi)}^Nd\xi \leq \frac{N^{\frac{1}{\alpha}}}{2\pi}\pa{1- \beta_N^{2+\tau} + \phi_{\tau}(\beta_N)}^{N-q} \Norm{\widehat{g}}^q_{L^q(\mathbb{R})}.
\end{gathered}
\end{equation}
Lastly, we need to estimate $I_3$, which is the simplest of the three integrals. Indeed
\begin{equation}\label{eq: estimation of I_3}
\begin{gathered}
I_3 = \frac{1}{2\pi}\int_{\abs{\xi}>\beta_N N^{\frac{1}{\alpha}}}e^{-\sigma\abs{\xi}^\alpha}d\xi \leq \frac{e^{-\frac{\sigma N \beta_N^\alpha}{2}}}{2\pi}\int_{\abs{\xi}>\beta_N N^{\frac{1}{\alpha}}}e^{-\frac{\sigma\abs{\xi}^\alpha}{2}}d\xi\\
\leq D e^{-\frac{\sigma N \beta_N^\alpha}{2}},
\end{gathered}
\end{equation}
where $D=\frac{1}{2\pi}\int_{\mathbb{R}}e^{-\frac{\sigma\abs{\xi}^\alpha}{2}}d\xi$. Combining (\ref{eq: estimation of I_1}), (\ref{eq: estimation of I_2}) and (\ref{eq: estimation of I_3}) yields the desired result.
\end{proof}
\begin{remark}\label{rem: epsilon N goes to zero explicitly, unlikefixed beta}
It is clear that if $\br{\beta_N}_{N\in\mathbb{N}}$ is chosen such that it goes to zero and
\begin{equation}\nonumber
\beta_N^{2+\tau}N \underset{N\rightarrow\infty}{\longrightarrow}\infty
\end{equation}
then $\epsilon_\tau(N)$, defined in the above theorem, goes to zero as $N$ goes to infinity, and we have an explicit rate to how fast it does it. A different method to undertake here is to pick $\beta_0$ small enough that all the steps of the proof the theorem work, and get that
\begin{equation}\nonumber
\begin{gathered}
\Norm {g_N-\gamma_{\sigma,\alpha,\beta}}_\infty \leq C_{g,\alpha}\Bigg(N^{\frac{1}{\alpha}}(1-\beta_0^{2+\tau}+\phi_\tau(\beta_0))^{N-q} 
+e^{-\frac{\sigma N\beta_0^\alpha}{2}} \\
+ \omega_{\eta}(\beta_0)+2\sigma \beta_0^\alpha \pa{1+\beta^2 \tan^2\pa{\frac{\pi\alpha}{2}}}\Bigg).
\end{gathered}
\end{equation}
Thus
\begin{equation}\nonumber
\limsup_{N\rightarrow\infty}\Norm {g_N-\gamma_{\sigma,\alpha,\beta}}_\infty \leq \lim_{\beta_0 \rightarrow 0}\pa{ \omega_{\eta}(\beta_0)+2\sigma \beta_0^\alpha \pa{1+\beta^2 \tan^2\pa{\frac{\pi\alpha}{2}}}}=0,
\end{equation}
proving the desired convergence, but losing the explicit $N$ dependency!
\end{remark}
An immediate corollary of Theorem \ref{thm: main approximation themroem} is the following:
\begin{theorem}\label{thm: main approximation theorem with conditions}
Let $g$ be the probability density function of a random real variable $X$. Assume that $g\in L^{p^\prime}(\mathbb{R})$ for some $p^\prime>1$ and
\begin{enumerate}[(1)]
\item $\int |x|g(x)dx<\infty$.
\item $\mu_g(x)\underset{x\rightarrow\infty}{\sim} C_S x^{2-\alpha}$ for some $C_S>0$ and $1<\alpha < 2$ where 
\begin{equation}\nonumber
\mu_g(x)=\int_{-x}^x y^2g(y)dy.
\end{equation}
\item \begin{equation}\nonumber
\begin{gathered}
\frac{1-G(x)}{1-G(x)+G(-x)}\underset{x\rightarrow\infty}{\longrightarrow}p \\
\frac{G(-x)}{1-G(x)+G(-x)}\underset{x\rightarrow\infty}{\longrightarrow}q, 
\end{gathered}
\end{equation}
where $G(x)=\int_{-\infty}^x g(y)dy$.
\end{enumerate}
Then, for any positive sequence $\br{\beta_N}_{N\in\mathbb{N}}$ that converges to zero as $N$ goes to infinity and satisfies
\begin{equation}\label{eq: beta tau condition}
\beta_N^{2+\tau}N\underset{N\rightarrow\infty}{\longrightarrow}\infty,
\end{equation}
for some $\tau>0$ and for $N$ large enough, we have that
\begin{equation}\label{eq: main approxiamtion equation for convolution}
\begin{gathered}
\sup_x \abs{g^{\ast N}(x)-\frac{\gamma_{\sigma,\alpha,\beta}\pa{\frac{x-NE}{N^{\frac{1}{\alpha}}}}}{N^{\frac{1}{\alpha}}}} \leq \frac{C_{g,\alpha}}{N^{\frac{1}{\alpha}}}\Bigg(N^{\frac{1}{\alpha}}(1-\beta^{2+\tau}_N+\phi_\tau(\beta_N))^{N-q^\prime} \\
+e^{-\frac{\sigma N\beta_N^\alpha}{2}} 
+ \omega_{\eta}(\beta_N)+2\sigma \beta_N^\alpha \pa{1+\beta^2 \tan^2\pa{\frac{\pi\alpha}{2}}}\Bigg)=\frac{\epsilon_{\tau}(N)}{N^{\frac{1}{\alpha}}},
\end{gathered}
\end{equation}
where 
\begin{enumerate}[(i)]
\item $\sigma=C_S\frac{\Gamma(3-\alpha)}{\alpha(\alpha-1)}\cos\pa{\frac{\pi\alpha}{2}}$, $\beta=p-q$.
\item $E=\int_{\mathbb{R}}xg(x)dx$.
\item $C_{g,\alpha}>0$ is a constant depending only on $g$, its moments and $\alpha$.
\item  $q^\prime$ can be chosen to be the H\"older conjugate of $\min(2,p^{\prime})$.
\item  $\phi_{\tau}$ satisfies
\begin{equation}\nonumber
\lim_{x\rightarrow 0}\frac{\phi_{\tau}(x)}{|x|^{2+\tau}}=0,
\end{equation}
\item $\eta(\xi)$ is the reminder function of $e^{-i\xi E}\widehat{g}(\xi)$, defined in Definition \ref{def: FDA}, and $\omega_\eta(\beta)=\sup_{|x|\leq \beta}\frac{\abs{\eta(x)}}{\abs{x}^\alpha}$.
\end{enumerate}
Under the condition (\ref{eq: beta tau condition}) and the conclusions $(i)-(vi)$ one finds that
\begin{equation}\nonumber
 \lim_{N\rightarrow\infty}\epsilon_{\tau}(N)=0.
\end{equation}
\end{theorem}
\begin{proof}
We start by defining $g_0(x)=g(x+E)$. Clearly $g_0\in L^{p\prime}(\mathbb{R})$ and $\int_{\mathbb{R}}|x|g_0(x)dx < \infty$. If we will be able to show that $g_0$ is in the NDA of $\widehat{\gamma}_{\sigma,\alpha,\beta}$, then, using Theorem \ref{thm: main approximation themroem}, we can conclude that
\begin{equation}\nonumber
\begin{gathered}
\sup_x \abs{g^{\ast N}\pa{N^{\frac{1}{\alpha}}x+NE} - \frac{\gamma_{\sigma,\alpha,\beta}(x)}{N^{\frac{1}{\alpha}}}} \leq \frac{\epsilon_\tau(N)}{N^{\frac{1}{\alpha}}},
\end{gathered}
\end{equation}
as $g_{0}^{\ast N}(x)=g^{\ast N}(x+NE)$, and the desired result follows.\\
We only have to prove that $g_0$ is in the appropriate NDA. To do that we will use Theorem \ref{thm: NDA conditions from Feller}. From its definition we know that $g_0$ has zero mean. Clearly
\begin{equation}\nonumber
\begin{gathered}
\frac{1-G_0(x)}{1-G_0(x)+G_0(-x)}\underset{x\rightarrow\infty}{\longrightarrow}p \\
\frac{G_0(-x)}{1-G_0(x)+G_0(-x)}\underset{x\rightarrow\infty}{\longrightarrow}q, 
\end{gathered}
\end{equation}
with $G_0(x)=\int_{-\infty}^x g_0(y)dy$, as $G_0(x)=G(x+E)$. \\
Next, we see that
\begin{equation}\nonumber
\mu_{g_0}(x)=\int_{-x}^x y^2g_0(y)dy = \int_{-x+E}^{x+E}y^2 g(y)dy - 2E\int_{-x+E}^{x+E}yg(y)dy + E^2\int_{-x+E}^{x+E}g(y)dy.
\end{equation}
The first term is bounded between $\mu_g(x-E)$ and $\mu_g(x+E)$ and as such behaves like $C_S x^{2-\alpha}$ as $x$ goes to infinity. The rest of the terms have a limit as $x$ goes to infinity, implying that
\begin{equation}\nonumber
\mu_{g_0}(x)\sim C_S x^{2-\alpha}.
\end{equation}
All the conditions of Theorem \ref{thm: NDA conditions from Feller} are satisfied (see Remark \ref{rem: our particular case with respect to feller's thm}), with $\sigma$ and $\beta$ given by $(i)$, and the proof is complete.
\end{proof}
Now that we have an appropriate local central limit theorem, we are ready to go to the next section where we will show that families of the type
\begin{equation}\nonumber
F_N=\frac{f^{\otimes N}}{\mathcal{Z}_N\pa{f,\sqrt{N}}}
\end{equation}
are chaotic and entropically chaotic, for a large class of functions $f$ with moments of order $2\alpha$, $1<\alpha<2$.\\
Before we do that we'd like to mention that with additional conditions on $g$, the estimation on $\epsilon_\tau$, defined in Theorem \ref{thm: main approximation theorem with conditions}, can become more explicit. This will be done via an explicit estimation for $\omega_\eta(\xi)$. Such estimation can be found in \cite{GJT}, yet the additional conditions are very restrictive and we weren't able to find many functions that will satisfy all of them with our simpler conditions. As it is still of interest we will provide some information on the matter in the Appendix.

\section{Chaoticity and Entropic Chaoticity for Families with Unbounded Fourth Moment.}\label{sec: chaoticity and entropic chaoticity}
The study of the chaoticity and entropic chaoticity of distribution function, $\br{F_N}_{N\in\mathbb{N}}$, on Kac's sphere that have the special form given in (\ref{eq: usual suspect}) is intimately connected to the asymptotic behaviour of the normalisation function $\mathcal{Z}_N\pa{f,r}$ at all $r$, and not only its value at $r=\sqrt{N}$. Formula (\ref{eq: interpretation of the normalisation function}) for the normalisation function, presented in Section \ref{sec: preliminaries}, and the local central limit theorem we just proved provide us with the necessary tools to find the desired behaviour. 
\begin{theorem}\label{thm: asyptotic behaviour for h conv N and normalisation function}
Let $f$ be the probability density function of a random real variable $V$ such that $f\in L^p(\mathbb{R})$ for some $p>1$. Let
\begin{equation}\nonumber
\nu_{f}(x)=\int_{-\sqrt{x}}^{\sqrt{x}}y^4 f(y)dy,
\end{equation}
and assume that
\begin{equation}\nonumber
\int_{\mathbb{R}}x^2 f(x)dx=E<\infty.
\end{equation}
and $\nu_{f}(x)\underset{x\rightarrow\infty}{\sim}C_S x^{2-\alpha}$ for some $C_S>0$ and $1<\alpha<2$. Then
\begin{equation}\label{eq: asymptote for h conv N}
\sup_x \abs{h^{\ast N}(x)-\frac{\gamma_{\sigma,\alpha,1}\pa{\frac{x-NE}{N^{\frac{1}{\alpha}}}}}{N^{\frac{1}{\alpha}}}}\leq \frac{\epsilon(N)}{N^{\frac{1}{\alpha}}},
\end{equation}
where $\lim_{N\rightarrow\infty}\epsilon(N)=0$, $\sigma=C_S \frac{\Gamma (3-\alpha)}{\alpha(\alpha-1)}\cos\pa{\frac{\pi\alpha}{2}}$ and $h$ is the probability density function of the random variable $V^2$. Moreover, $\epsilon(N)$ can be bound by $\epsilon_{\tau}(N)$, given in Theorem \ref{thm: main approximation theorem with conditions}, with $\eta$ the reminder function of $\widehat{h}$. \\
In addition,
\begin{equation}\label{eq: asymptote for normalisation function}
\mathcal{Z}_N\pa{f,\sqrt{r}}=\frac{2}{\abs{\mathbb{S}^{N-1}}r^{\frac{N-2}{2}}}\frac{1}{N^{\frac{1}{\alpha}}}\pa{\gamma_{\sigma,\alpha,1}\pa{\frac{r-NE}{N^{\frac{1}{\alpha}}}}+\lambda_N(r)},
\end{equation}
where $\sup_u \abs{\lambda_N(u)}\underset{N\rightarrow\infty}{\longrightarrow}0$.
\end{theorem}
\begin{proof}
We start by noticing that (\ref{eq: asymptote for normalisation function}) follows immediately from (\ref{eq: interpretation of the normalisation function}) and (\ref{eq: asymptote for h conv N}). Next, we will show that the conditions of Theorem \ref{thm: main approximation theorem with conditions} are satisfied by $h$, concluding inequality (\ref{eq: asymptote for h conv N}), and the estimation for $\epsilon(N)$.\\ 
As was mentioned before, the function $h$ is given by
\begin{equation}\nonumber
h(x)=\begin{cases}
\frac{f\pa{\sqrt{x}}+f\pa{-\sqrt{x}}}{2\sqrt{x}} & x>0 \\
0 & x\leq 0
\end{cases}
\end{equation}
and $h\in L^{p^\prime}(\mathbb{R})$ for some $p^\prime >1$ when $f\in L^p(\mathbb{R})$ with $p>1$ (see Remark \ref{rem: properties of h}). Moreover, for any $\kappa>0$
\begin{equation}\nonumber
\int_{\mathbb{R}}|x|^\kappa h(x)dx=\int_{\mathbb{R}}|x|^{2\kappa}f(x)dx,
\end{equation}
from which we conclude that 
\begin{equation}\nonumber
\int_{\mathbb{R}}|x|h(x)dx = \int_{\mathbb{R}}xh(x)=E<\infty.
\end{equation}
By its definition
\begin{equation}\nonumber
\begin{gathered}
\mu_h(x) = \int_{-x}^{x}y^2 h(y)dy 
=\nu_f(x)\underset{x\rightarrow\infty}{\sim}C_S x^{2-\alpha},
\end{gathered}
\end{equation}
and recalling Remark \ref{rem: our particular case with respect to feller's thm}, we conclude that if $H$ is the probability distribution function of $V^2$ then for any $x>0$
\begin{equation}\nonumber
\begin{gathered}
\frac{1-H(x)}{1-H(x)+H(-x)}=1 \\
\frac{H(-x)}{1-H(x)+H(-x)}=0.
\end{gathered}
\end{equation}
Thus, all the condition of Theorem \ref{thm: main approximation theorem with conditions} are satisfied by $h$ with the appropriate $\sigma,\alpha$ and $\beta=1$, and the proof is complete.
\end{proof}
\begin{remark}\label{rem: normalisation function approximation theorem}
A couple of remarks:
\begin{itemize}
\item The formula for the normalisation function, $\mathcal{Z}_N$, depends heavily on $h^{\ast N}$, where $h$ is the distribution function of the random variable $V^2$. Any hope for a normal central limit theorem, let alone a local one, relies heavily on the finiteness of the variance of $h$, i.e. the fourth moment of $f$. This is exactly the reason why the fourth moment of $f$ plays such an important role in the theory. When $f$ lacks that condition, a thing that manifests itself via the function $\nu_f(x)$ in the above theorem, there is still something that can be said and our local central limit theorem comes into play by replacing the Gaussian with the stable laws.  
\item The parameter $\beta$ represents the skewness of the stable distribution. In general $\beta\in[-1,1]$ and the closer it is to $1$, the more right skewed the distribution is. The closer it gets to $-1$, the more left skewed the distribution is. Since our probability density function $h$ is supported on the positive real line, it is not surprising that we got that $\beta$ must be $1$! 
\end{itemize}
\end{remark}
We are now ready to prove Theorems \ref{thm: chaoticity and entropic chaoticity general} and \ref{thm: chaoticity and entropic chaoticity with growth conditions on f}.
\begin{proof}[Proof of Theorem \ref{thm: chaoticity and entropic chaoticity general}]
Due to the given information on $f$, we see that it satisfies all the conditions of Theorem \ref{thm: asyptotic behaviour for h conv N and normalisation function}, and as such for any finite $k\in\mathbb{R}$
\begin{equation}\label{eq: chaoticity and entropic chatoicity proof I}
\begin{gathered}
\abs{\mathbb{S}^{N-k-1}}r^{\frac{N-k-2}{2}}\mathcal{Z}_{N-k}\pa{f,\sqrt{r}} \\
=\frac{2}{\pa{N-k}^{\frac{1}{\alpha}}}\pa{\gamma_{\sigma,\alpha,1}\pa{\frac{r-\pa{N-k}}{\pa{N-k}^{\frac{1}{\alpha}}}}+\lambda_{N-k}(r)},
\end{gathered}
\end{equation}
for some $\sigma=C_S \frac{\Gamma(3-\alpha)}{\alpha(\alpha-1)}\cos\pa{\frac{\pi\alpha}{2}}$ and $\lambda_{N-k}$ such that
\begin{equation}\nonumber
\epsilon_{N-k}=\sup_r \abs{\lambda_{N-k}(r)}\underset{N\rightarrow\infty}{\longrightarrow}0.
\end{equation}
Using Lemma \ref{lem: k-th marginal} with $F_N=\frac{f^{\otimes N}}{\mathcal{Z}_N\pa{f,\sqrt{N}}}$ we find that
\begin{equation}\nonumber
\begin{gathered}
\Pi_k(F_N)\pa{v_1,\dots,v_k}=\frac{\abs{\mathbb{S}^{N-k-1}}\pa{N-\sum_{i=1}^k v_i^2}_{+}^{\frac{N-k-2}{2}}\mathcal{Z}_{N-k}\pa{f,\sqrt{N-\sum_{i=1}^k v_i^2}}}{\abs{\mathbb{S}^{N-1}}N^{\frac{N-2}{2}}\mathcal{Z}_N\pa{f,\sqrt{N}}}\\
\cdot f^{\otimes k}\pa{v_1,\dots,v_k}.
\end{gathered}
\end{equation}
Combining this with (\ref{eq: chaoticity and entropic chatoicity proof I}) yields
\begin{equation}\label{eq: chaoticity and entropic chatoicity proof II}
\begin{gathered}
\Pi_k(F_N)\pa{v_1,\dots,v_k}=\pa{\frac{N}{N-k}}^{\frac{1}{\alpha}}\frac{\gamma_{\sigma,\alpha,1}\pa{\frac{k-\sum_{i=1}^k v_i^2}{\pa{N-k}^{\frac{1}{\alpha}}}}+\lambda_{N-k}\pa{N-\sum_{i=1}^k v_i^2}}{\gamma_{\sigma,\alpha,1}(0)+\lambda_N(N)}\\
\cdot f^{\otimes k}\pa{v_1,\dots,v_k}\chi_{\sum_{i=1}^k v_i^2 \leq N}\pa{v_1,\dots,v_k},
\end{gathered}
\end{equation}
where $\chi_A$ is the characteristic function of the set $A$. By its definition, given in (\ref{eq: def of gamma_sigma}), and the properties of $\widehat{\gamma}_{\sigma,\alpha,\beta}$, we know that $\gamma_{\sigma,\alpha,1}$ is bounded and continuous on $\mathbb{R}$. As such, along with the conditions on $\lambda_{N-k}$ and $\lambda_N$, we conclude that
\begin{equation}\nonumber
\Pi_k(F_N)\pa{v_1,\dots,v_k} \underset{N\rightarrow\infty}{\longrightarrow}f^{\otimes k}\pa{v_1,\dots,v_k},
\end{equation}
pointwise. Using Lemma \ref{lem: simple chaoticity condition} we obtain that $\br{F_N}_{N\in\mathbb{N}}$ is $f-$chaotic.\\
Next we turn our attention to the entropic chaos. Using symmetry, (\ref{eq: chaoticity and entropic chatoicity proof I}) and (\ref{eq: chaoticity and entropic chatoicity proof II}) we find that
\begin{equation}\nonumber
\begin{gathered}
H_N(F_N)=\frac{1}{\mathcal{Z}_N\pa{f,\sqrt{N}}}\int_{\mathbb{S}^{N-1}\pa{\sqrt{N}}}f^{\otimes N}\log\pa{f^{\otimes N}}d\sigma^N - \log\pa{\mathcal{Z}_{N}\pa{f,\sqrt{N}}} \\
=N\int_{\mathbb{R}}\Pi_1(F_N)(v_1)\log \pa{f(v_1)}dv_1 - \log \pa{\frac{2\pa{\gamma_{\sigma,\alpha,1}(0)+\lambda_N(N)}}{\abs{\mathbb{S}}^{N-1}N^{\frac{N-2}{2}+\frac{1}{\alpha}}}} \\
=N\pa{\frac{N}{N-1}}^{\frac{1}{\alpha}} \int_{-\sqrt{N}}^{\sqrt{N}}\frac{\gamma_{\sigma,\alpha,1}\pa{\frac{1-v_1^2}{\pa{N-1}^{\frac{1}{\alpha}}}}+\lambda_{N-1}\pa{N-v_1^2}}{\gamma_{\sigma,\alpha,1}(0)+\lambda_N(N)}
f(v_1)\log f(v_1)dv_1 \\
- \log \pa{2\sqrt{\pi}\pa{\gamma_{\sigma,\alpha,1}(0)+\lambda_N(N)}\pa{1+O\pa{\frac{1}{N}}}}+\pa{\frac{1}{\alpha}-\frac{1}{2}}\log N+\frac{N}{2}\log\pa{2\pi e}.
\end{gathered}
\end{equation}
where we have used the fact that $\abs{\mathbb{S}^{N-1}}=\frac{2\pi^{\frac{N}{2}}}{\Gamma\pa{\frac{N}{2}}}$, and an asymptotic approximation for the Gamma function.\\ 
Since
\begin{equation}\nonumber
\begin{gathered}
\abs{\frac{\gamma_{\sigma,\alpha,1}\pa{\frac{1-v_1^2}{\pa{N-1}^{\frac{1}{\alpha}}}}+\lambda_{N-1}\pa{N-v_1^2}}{\gamma_{\sigma,\alpha,1}(0)+\lambda_N(N)}
f(v_1)\log f(v_1) } \\
\leq \frac{\Norm{\gamma_{\sigma,\alpha,1}}_\infty + \epsilon_{N-1}}{\gamma_{\sigma,\alpha,1}(0)-\epsilon_N}f(v_1)\abs{\log f(v_1)}  \\
\leq  \frac{2\pa{\Norm{\gamma_{\sigma,\alpha,1}}_\infty + 1}}{\gamma_{\sigma,\alpha,1}(0)}f(v_1)\abs{\log f(v_1)}\in L^1 (\mathbb{R}), 
\end{gathered}
\end{equation}
for $N$ large enough. Combining this with the fact that $\br{\Pi_1(F_N)}_{N\in\mathbb{N}}$ converges to $f$ pointwise, we can use the dominated convergence theorem to conclude that
\begin{equation}\label{eq: chaoticity and entropic chatoicity proof III}
\lim_{N\rightarrow\infty}\frac{H_N(F_N)}{N}= \int_\mathbb{R}f(v_1)\log f(v_1)dv_1 + \frac{\log 2\pi +1}{2}=H(f|\gamma),
\end{equation}
and the proof is complete.
\end{proof} 
\begin{proof}[Proof of Theorem \ref{thm: chaoticity and entropic chaoticity with growth conditions on f}]
It is easy to see that the condition  $f(x)\underset{x\rightarrow\infty}{\sim}\frac{D}{|x|^{1+2\alpha}}$ for some $1<\alpha<2$ and $D>0$ implies that 
\begin{equation}\nonumber
\nu_f(x)\underset{x\rightarrow\infty}{\sim}\frac{D}{2-\alpha} x^{2-\alpha}.
\end{equation}

Thus, with the added information given in the theorem we know that $f$ satisfies the conditions of Theorem \ref{thm: chaoticity and entropic chaoticity general}, and we conclude the desired result.
\end{proof}
\begin{remark}\label{rem: explict exmples of functions}
Theorem \ref{thm: chaoticity and entropic chaoticity with growth conditions on f} gives rise to many, previously unknown, entropically chaotic families, determined mainly by a simple growth condition. An explicit example to such family is the one generated by the function
\begin{equation}\nonumber
f(x)=\frac{\sqrt{2}}{\pi\pa{1+x^4}}.
\end{equation}
\end{remark}

\section{Lower Semi Continuity and Stability Property.}\label{sec: lsc and stability}
As discussed in Section \ref{sec: introduction}, the concept of entropic chaoticity is much stronger than that of normal chaoticity. This is due to the inclusion of all correlation information and an appropriate rescaling of the relative entropy. In this section we will show that the rescaled entropy is a good form of distance, one that is stable under certain conditions.\\
The first step we must make, inspired by \cite{CCRLV}, is a form of lower semi continuity property for the relative entropy on Kac's sphere, expressed in Theorem \ref{thm: lower semi continuity}. To begin with, we mention that in \cite{CCRLV}, the authors have proved the following:
\begin{theorem}\label{thm: CCRLV theorem}
Let $g$ be a probability density function on $\mathbb{R}$ such that $g\in L^p(\mathbb{R})$ for some $p>1$. Assume in addition that
\begin{equation}\nonumber
\int_{\mathbb{R}}x^2 g(x)dx=1, \quad \int_{\mathbb{R}}x^4 g(x)dx<\infty,
\end{equation}
and denote $d\nu_N=G_N d\sigma^N$, where $G_N=\frac{g^{\otimes N}}{\mathcal{Z}_N \pa{g,\sqrt{N}}}$, restricted to Kac's sphere. Let $\br{\mu_N}_{N\in\mathbb{N}}$ be a family of symmetric probability measures on Kac's sphere such that for some $k\in\mathbb{N}$ we have that
\begin{equation}\nonumber
\Pi_k(\mu_N) \underset{N\rightarrow\infty}{\rightharpoonup}\mu_k.
\end{equation}
Then
\begin{equation}\label{eq: l.s.c in CCRLV}
\frac{H\pa{\mu_k | g^{\otimes k}}}{k} \leq \liminf_{N\rightarrow\infty}\frac{H_N(\mu_N | \nu_N)}{N}.
\end{equation}
\end{theorem} 
Note that due to an inequality, the so-called Csiszar-Kullback-Leibler-Pinsker inequality (\cite{CKLP}) ,one has that
\begin{equation}\label{eq: kullbak pinsher}
\Norm{\mu-\nu}_{TV} \leq \sqrt{2H(\mu|\nu)}, 
\end{equation}
showing that (\ref{eq: l.s.c in CCRLV}) gives a stronger result than an $L^1$ convergence. We will use this theorem as a motivation for our lower semi continuity property, as well as in the particular case of
\begin{equation}\nonumber
g(x)=\gamma(x), \quad d\nu_N=G_Nd\sigma^N=d\sigma^N, 
\end{equation}
where $\gamma(x)$ is the standard Gaussian.\\
Before we begin the proof of Theorem \ref{thm: lower semi continuity} we point out the obvious difference between the $k=1$ and $k>1$ cases. This is due to the fact that the proof relies heavily on our approximation theorem, Theorem \ref{thm: asyptotic behaviour for h conv N and normalisation function}, which is valid \emph{only} in one dimension. The higher dimension case needs to be tackled differently, unlike the proof of Theorem \ref{thm: CCRLV theorem}, where the higher dimension case is proven in a very similar way.\\
The proof of Theorem \ref{thm: lower semi continuity} follows ideas presented in \cite{CCRLV}, with some modification to our current discussion.
\begin{proof}[Proof of Theorem \ref{thm: lower semi continuity}]
We start by noticing that since $C_b\pa{\mathbb{R}^{k_0}}$ can be considered a subspace of $C_b\pa{\mathbb{R}^k}$ whenever $k_0 \leq k$. The weak convergence condition on $\Pi_k(\mu_N)$ implies that
\begin{equation}\nonumber
\Pi_{k_0}(\mu_N)\underset{N\rightarrow\infty}{\rightharpoonup}\mu_{k_0}=\Pi_{k_0}(\mu_k).
\end{equation}
In particular we find that $\Pi_1(\mu_N)$ converges weakly to $\mu=\Pi_1(\mu_k)$.\\
Next, we recall a duality formula for the relative entropy (see \cite{Lott} for instance, for the compact case):
\begin{equation}\label{eq: duality formula for relative entropy}
H(\mu | \nu) = \sup_{\varphi \in C_b} \left\lbrace \int \varphi d\mu - \log\left(\int e^{\varphi}d\nu \right)\right\rbrace.
\end{equation}
Given $\epsilon>0$ we can find $\varphi_{\epsilon}\in C_b(\mathbb{R})$ such that
\begin{equation}\nonumber
\int_{\mathbb{R}}e^{\varphi_{\epsilon}(v)}f(v)dv = 1
\end{equation}
and
\begin{equation}
H(\mu | f) \leq \int_{\mathbb{R}}\varphi_{\epsilon}(v)d\mu(v) +\frac{\epsilon}{2}.
\end{equation}
We can find a compact set $K_{\epsilon}\subset \mathbb{R}$ such that
\begin{equation}\nonumber
\mu\pa{K^c_\epsilon} \leq \frac{\epsilon}{4\Norm{\varphi_\epsilon}_\infty}, \quad \int_{K_\epsilon^c}f(v)dv \leq \frac{\epsilon}{2e^{\Norm{\varphi_\epsilon}_\infty}}.
\end{equation}
Let $\eta_{\epsilon}\in C_c(\mathbb{R})$ be such that
\begin{equation}\nonumber
0\leq \eta_{\epsilon} \leq 1, \quad \eta_{\epsilon}|_{K_{\epsilon}}=1,
\end{equation}
and define $\varphi(v)=\eta_{\epsilon}(v)\varphi_{\epsilon}(v)$. Clearly $\varphi\in C_c(\mathbb{R})$, $\abs{\varphi}\leq \abs{\varphi_\epsilon}$ and
\begin{equation}\label{eq: lsc proof I}
H(\mu | f) \leq \int_{\mathbb{R}}\varphi(v)d\mu(v)+ 2\Norm{\varphi_{\epsilon}}_\infty \mu \pa{K^c_\epsilon} +\frac{\epsilon}{2}<\int_{\mathbb{R}}\varphi(v)d\mu(v)+\epsilon.
\end{equation}
Also,
\begin{equation}\label{eq: lsc proof II}
\abs{\int_{\mathbb{R}}e^{\varphi(v)}f(v)dv-\int_{\mathbb{R}}e^{\varphi_\epsilon(v)}f(v)dv} \leq 2e^{\Norm{\varphi_\epsilon}_\infty}\int_{K_\epsilon^c}f(v)dv < \epsilon.
\end{equation}
For any $N\in\mathbb{N}$, define $\phi_{N}\pa{v_1,\dots,v_N}=\sum_{i=1}^N \varphi(v_i)\in C_b\pa{\mathbb{R}^N}$. Plugging $\phi_N$ as a candidate in (\ref{eq: duality formula for relative entropy}), in the setting of Kac's sphere, and using symmetry we find that
\begin{equation}\nonumber
\begin{gathered}
H_N(\mu_N | \nu_N) \geq N\int_{\mathbb{R}}\varphi(v_1)d\Pi_{1}(\mu_N)(v_1) - \log\pa{\frac{1}{\mathcal{Z}_N\pa{f,\sqrt{N}}}\int_{\mathbb{S}^{N-1}\pa{\sqrt{N}}}\Pi_{i=1}^N\pa{e^{\varphi(v_i)}f(v_i)}d\sigma^N}\\
=N\int_{\mathbb{R}}\varphi(v_1)d\Pi_{1}(\mu_N)(v_1) - \log\pa{\frac{\mathcal{Z}_N \pa{\frac{e^{\varphi}f}{a},\sqrt{N}}}{\mathcal{Z}_N\pa{f,\sqrt{N}}}}-N\log a,
\end{gathered}
\end{equation}
where $a=\int_{\mathbb{R}}e^{\varphi(v)}f(v)dv$. Since $f$ satisfies the conditions of Theorem \ref{thm: asyptotic behaviour for h conv N and normalisation function}, so does the probability density function $\frac{e^\varphi}{a}f$. Denoting by $E=\frac{1}{a}\int_{\mathbb{R}}v^2 e^{\varphi(v)}f(v)dv$ we find that
\begin{equation}\label{eq: lsc proof III}
\frac{\mathcal{Z}_N\pa{\frac{e^{\varphi}f}{a},\sqrt{N}}}{\mathcal{Z}_N(f,\sqrt{N})}=\frac{\gamma_{\sigma_1,\alpha,1}\left(\frac{N-NE}{N^{\frac{1}{\alpha}}} \right)+\epsilon_1(N)}{\gamma_{\sigma,\alpha,1}\left( 0 \right)+\epsilon_2(N)},
\end{equation}
for some $\sigma,\sigma_1$, and $\br{\epsilon_i(N)}_{i=1,2}$ that go to zero as $N$ goes to infinity. Since $\gamma_{\sigma_1,\alpha,1}$ is the defined as the inverse Fourier transform of an $L^1$ function we know that
\begin{equation}\nonumber
\lim_{|x|\rightarrow\infty}\gamma_{\sigma_1,\alpha,1}(x)=0.
\end{equation}
Thus,
\begin{equation}\label{eq: lsc proof IV}
\liminf_{N\rightarrow\infty}\pa{-\frac{\log\pa{\gamma_{\sigma_1,\alpha,1}\left(\frac{N-NE}{N^{\frac{1}{\alpha}}} \right)+\epsilon_1(N)}}{N}} \geq 0.
\end{equation}
Together with the fact that
\begin{equation}\nonumber
\lim_{N\rightarrow\infty}\pa{-\frac{\log\pa{\gamma_{\sigma,\alpha,1}(0)+\epsilon_2(N)}}{N}} = 0,
\end{equation}
the weak convergence of $\Pi_1(\mu_N)$ and (\ref{eq: lsc proof I}), we find that
\begin{equation}\label{eq: lsc proof V}
\begin{gathered}
\liminf_{N\rightarrow\infty}\frac{H_N(\mu_N | \nu_N)}{N} \geq \int_{\mathbb{R}}\varphi(v)d\mu(v) -\log (1+\epsilon)\\
 \geq H(\mu | f) -\epsilon-\log(1+\epsilon),
\end{gathered}
\end{equation}
where we have used (\ref{eq: lsc proof II}) to conclude that $\abs{a-1} < \epsilon$. Since $\epsilon$ was arbitrary, $(i)$ is proved.\\
In order to show $(ii)$, we notice that
\begin{equation}\nonumber
\begin{gathered}
H_N(\mu_N | \nu_N)=\int_{\mathbb{S}^{N-1}\pa{\sqrt{N}}}\log\pa{\frac{d\mu_N}{F_N d\sigma^N}}d\mu_N =H_N(\mu_N | \sigma^N) 
- \int_{\mathbb{S}^{N-1}\pa{\sqrt{N}}}\log \pa{F_N}d\mu_N \\= H_N(\mu_N | \sigma^N) - N\int_{\mathbb{R}}\log\pa{f(v_1)}d\Pi_1(\mu_N)+\log\pa{\mathcal{Z}_N\pa{f,\sqrt{N}}}.
\end{gathered}
\end{equation}
Thus, for any $\delta>0$,
\begin{equation}\label{eq: lsc proof VI}
\begin{gathered}
\liminf_{N\rightarrow\infty}\frac{H_N(\mu_N | \nu_N)}{N}+ \limsup_{N\rightarrow\infty}\int_{\mathbb{R}}\log\pa{f(v_1)+\delta}d\Pi_1(\mu_N) \\
\geq \liminf_{N\rightarrow\infty}\frac{H_N(\mu_N | \sigma^N)}{N}-\frac{\log(2\pi)+1}{2},
\end{gathered}
\end{equation}
where we have used the fact that $\lim_{N\rightarrow\infty}\frac{\log\pa{\mathcal{Z}_N\pa{f,\sqrt{N}}}}{N}=-\frac{\log(2\pi)+1}{2}$, shown in the proof of Theorem \ref{thm: chaoticity and entropic chaoticity general}. From Theorem \ref{thm: CCRLV theorem} we know that 
\begin{equation}\nonumber
\liminf_{N\rightarrow\infty}\frac{H_N(\mu_N | \sigma^N)}{N} \geq \frac{H(\mu_k | \gamma^{\otimes k})}{k},
\end{equation}
and since 
\begin{equation}\nonumber
\begin{gathered}
H(\mu_k |f^{\otimes k} )=H(\mu_k | \gamma^{\otimes k})+ \int_{\mathbb{R}^k}\log\pa{\frac{\gamma^{\otimes k}}{f^{\otimes k}}}d\mu_k \\
=H(\mu_k | \gamma^{\otimes k}) -\frac{k\pa{\log(2\pi)+\int_{\mathbb{R}}v^2d\mu(v)}}{2}- k \int_{\mathbb{R}}\log\pa{f(v)}d\mu(v)
\end{gathered}
\end{equation}
we get the desired result from (\ref{eq: lsc proof VI}).
\end{proof}
We will now prove our first stability result, Theorem \ref{thm: stability property}. Again, the ideas presented here are motivated by \cite{CCRLV}.
\begin{proof}[Proof of Theorem \ref{thm: stability property}]
We start with the simple observation that if $\br{\mu_N}_{N\in\mathbb{N}}$ is a family of symmetric probability measures on Kac's sphere then $\br{\Pi_k(\mu_N)}_{N\in\mathbb{N}}$ is a tight family, for any $k\in\mathbb{N}$. Indeed, given $k\in \mathbb{N}$ we can find $m_N,r_N\in\mathbb{N}$ such that
\begin{equation}\nonumber
N=m_Nk+r_N,
\end{equation} 
where $0\leq r_N <k$. We have that
\begin{equation}\nonumber
\begin{gathered}
\Pi_k(\mu_N)\pa{\br{\sqrt{\sum_{i=1}^k v_i^2 }>R}} \leq \frac{1}{R^2}\int_{\sum_{i=1}^k v_i^2 >R^2}\pa{\sum_{i=1}^k v_i^2}d\Pi_k(\mu_N) \\
\leq \frac{1}{m_N R^2}\int_{\mathbb{S}^{N-1}\pa{\sqrt{N}}}\pa{\sum_{i=1}^{m_Nk} v_i^2}d\mu_N \leq \frac{N}{m_N R^2} < \frac{2k}{R^2},
\end{gathered}
\end{equation}
proving the tightness.\\
Since $\br{\Pi_1{\mu_N}}_{N\in\mathbb{N}}$ is tight, we can find a subsequence, $\br{\Pi_1\pa{\mu_{N_{k_j}}}}_{j\in\mathbb{N}}$, to any subsequence $\br{\Pi_1\pa{\mu_{N_k}}}_{k\in\mathbb{N}}$, that converges to a limit. Denote by $\kappa$ the weak limit of such one subsequence. Using (\ref{eq: l.s.c. for k=1}) we conclude that
\begin{equation}\label{eq: stability I}
H(\kappa | f) \leq \liminf_{j\rightarrow\infty}\frac{H_{N_{k_j}}\pa{\mu_{N_{k_j}}| \nu_{N_{k_j}}}}{N_{k_j}} =0,
\end{equation}
due to condition (\ref{eq: entropic closeness}). Thus, $\kappa=f(v)dv$, and since $\kappa$ was an arbitrary weak limit, we conclude that all possible weak limit points must be $f(v)dv$. Since the weak topology on $P(\mathbb{R})$ is metrisable we conclude that 
\begin{equation}\nonumber
\Pi_1(\mu_N)\underset{N\rightarrow\infty}{\rightharpoonup}f(v)dv=\mu.
\end{equation}
We will show that the convergence is actually in $L^1$ with the weak topology. \\
As an intermediate step in the proof of Theorem \ref{thm: lower semi continuity} we have shown that
\begin{equation}\label{eq: stability I.5}
\begin{gathered}
H(\mu_N | \nu_N) = H(\mu_N | \sigma^N) - N \int_{\mathbb{R}}\log\pa{f(v_1)}d\Pi_1(\mu_N)(v_1) \\
+ \log \pa{\mathcal{Z}_N\pa{f,\sqrt{N}}}.
\end{gathered}
\end{equation}
Using condition (\ref{eq: entropic closeness}), the fact that $\lim_{N\rightarrow\infty}\frac{\mathcal{Z}_N\pa{f,\sqrt{N}}}{N}=-\frac{\log(2\pi)+1}{2}$, and the fact that $f\in L^\infty (\mathbb{R})$ we conclude that there exists $C>0$, independent of $N$, such that for any $\delta>0$
\begin{equation}\label{eq: stability II}
\frac{H(\mu_N | \sigma_N)}{N} \leq C + \log \pa{\Norm{f}_\infty + \delta}.
\end{equation}
The inequality
\begin{equation}\nonumber
\frac{H\pa{\Pi_k(\mu_N) | \Pi_k(\sigma^N)}}{k} \leq 2 \frac{H_N(\mu_N | \sigma^N)}{N}
\end{equation}
proven in \cite{BCM} and valid for any $k\geq 1$ and $N\geq k$, implies that
\begin{equation}\label{eq: stability III}
H\pa{\Pi_k(\mu_N) | \Pi_k(\sigma^N)} \leq 2k\pa{C+\log\pa{\Norm{f}_\infty}+\delta},
\end{equation}
for all $k\in \mathbb{N}$, $N\geq k$ and $\delta>0$.\\
Similar to the proof of Theorem \ref{thm: lower semi continuity}, one can easily see that
\begin{equation}\label{eq: stability IV}
H\pa{\Pi_k(\mu_N) | \gamma^{\otimes k}}=H\pa{\Pi_k(\mu_N) | \Pi_k(\sigma^N)} + \int_{\mathbb{R}^k}\log\pa{\frac{\Pi_k(\sigma^N)}{\gamma^{\otimes k}}}d\Pi_k(\mu_N)
\end{equation}
where $\gamma$ is the standard Gaussian. Since $d\sigma^N=\frac{\gamma^{\otimes N}}{\mathcal{Z}_N\pa{\gamma,\sqrt{N}}}d\sigma^N$, and $\gamma$ is a probability density with finite fourth moment, one can employ similar theorems to those presented here and find that
\begin{equation}\nonumber
\frac{\Pi_k(\sigma^N)\pa{v_1,\dots,v_k}}{\gamma^{\otimes k}\pa{v_1,\dots,v_k}}=\sqrt{\frac{N}{N-k}}\cdot\frac{\gamma\pa{\frac{k-\sum_{i=1}^k v_i^2 }{\sqrt{2N}}}+\lambda_{N-k}\pa{N-k-\sum_{i=1}^k v_I^2}}{1+\lambda_N(N)}\chi_{\sum_{i=1}^k v_i^2 \leq N},
\end{equation}
where $\sup_u\abs{\lambda_{N-k}(u)}\underset{N\rightarrow\infty}{\longrightarrow}0$ and $\lambda_N(N)\underset{N\rightarrow\infty}{\longrightarrow}0$ (see \cite{CCRLV} for more details). As such, 
\begin{equation}\nonumber
\int_{\mathbb{R}^k}\log\pa{\frac{\Pi_k(\sigma^N)}{\gamma^{\otimes k}}}d\Pi_k(\mu_N) \leq \log\pa{\max_{N>k}\sqrt{\frac{N}{N-k}}\frac{\Norm{\gamma}_\infty + \sup_N \sup_u \abs{\lambda_{N-k}(u)}}{1+\inf_N \lambda_N(N)}},
\end{equation}
which, together with (\ref{eq: stability III}) and (\ref{eq: stability IV}) shows that
\begin{equation}\nonumber
H\pa{\Pi_k(\mu_N) | \gamma^{\otimes k}} \leq 2k\pa{C+\log\pa{\Norm{f}_\infty}+\delta}+D,
\end{equation}
for some $C,D>0$ independent of $N$, and $\delta>0$. Thus, $\br{\Pi_k{\mu_N}}_{N\in\mathbb{N}}$ has bounded relative entropy with respect to $\gamma^{\otimes k}$ and we can apply the Dunford-Pettis compactness theorem and conclude that the densities of $\br{\Pi_k(\mu_N)}_{N\in\mathbb{N}}$ form a relatively compact set in $L^1(\mathbb{R}^k)$ with the weak topology. Since this is true for all $k$, and we know that $\br{\Pi_1(\mu_N)}_{N\in\mathbb{N}}$ converge weakly (in the measure sense) to $\mu$, with density function $f(v)$, we conclude that for any $\phi\in L^\infty(\mathbb{R})$ we have that
\begin{equation}\label{eq: stability V}
\int_{\mathbb{R}}\phi(v)d\Pi_1(\mu_N)(v) \underset{N\rightarrow\infty}{\longrightarrow}\int_{\mathbb{R}}\phi(v)f(v)dv.
\end{equation}
In particular, since $f\in L^{\infty}(\mathbb{R})$ and $f\geq 0$ we have that for any $\delta>0$
\begin{equation}\label{eq: stability VI}
\int_{\mathbb{R}}\log\pa{f(v)+\delta}d\Pi_1(\mu_N)(v) \underset{N\rightarrow\infty}{\longrightarrow}\int_{\mathbb{R}}\log\pa{f(v)+\delta}f(v)dv.
\end{equation}
Combining (\ref{eq: stability VI}), (\ref{eq: entropic closeness}) with the fact that $\Pi_1(\mu_N)$ converges to $f(v)dv$, we find that if $\br{\Pi_k\pa{\mu_{N_j}}}_{j\in\mathbb{N}}$ converges weakly to $\kappa_k$, then by (\ref{eq: semi lsc for k>1})
\begin{equation}\label{eq: stability VII}
\frac{H(\kappa_k | f^{\otimes k})}{k} \leq \int_{\mathbb{R}}\log \pa{f(v)+\delta}f(v)dv - \int_{\mathbb{R}}\log \pa{f(v)}f(v)dv
\end{equation}
where we have used the fact that $\int_{\mathbb{R}}v^2d\mu(v)=\int_{\mathbb{R}}v^2f(v)dv=1$. Using the dominated convergence theorem to take $\delta$ to zero shows that $H(\kappa_k | f^{\otimes k})=0$, and so
\begin{equation}\nonumber
\kappa_k=f^{\otimes k}\pa{v_1,\dots,v_k}dv_1\dots dv_k.
\end{equation}
Much like $\br{\Pi_1(\mu_N)}_{N\in\mathbb{N}}$, since $\br{\Pi_k(\mu_N)}_{N\in\mathbb{N}}$ is tight we can always find weak limits for some subsequences of it. We have just proved that all possible weak limits of subsequences of $\br{\Pi_k(\mu_N)}_{N\in\mathbb{N}}$ are $f^{\otimes k}$, from which we conclude that 
\begin{equation}\nonumber
\Pi_k(\mu_N) \underset{N\rightarrow\infty}{\rightharpoonup}f^{\otimes k},
\end{equation}
showing the chaoticity. It is worth to note that we actually proved more than the above: we have proved convergence in $L^1(\mathbb{R}^k)$ with the weak topology.\\
Going back to (\ref{eq: stability I.5}), and using (\ref{eq: entropic closeness}), (\ref{eq: stability VI}) and the known limit of $\frac{\log\pa{\mathcal{Z}_N\pa{f,\sqrt{N}}}}{N}$ we find that
\begin{equation}
\limsup_{N\rightarrow\infty}\frac{H_N(\mu_N | \sigma^N)}{N} \leq \int_{\mathbb{R}}\log\pa{f(v)+\delta}f(v)dv + \frac{\log(2\pi)+1}{2}.
\end{equation}\nonumber
Taking $\delta$ to zero we conclude that
\begin{equation}\label{eq: stability VIII}
\limsup_{N\rightarrow\infty}\frac{H_N(\mu_N | \sigma^N)}{N} \leq H(f | \gamma).
\end{equation}
Since the inequality
\begin{equation}\nonumber
\liminf_{N\rightarrow\infty}\frac{H_N(\mu_N | \sigma^N)}{N} \geq H(f | \gamma)
\end{equation}
follows from Theorem \ref{thm: CCRLV theorem}, we see that
\begin{equation}\label{eq: stability IX}
\lim_{N\rightarrow\infty}\frac{H_N(\mu_N | \sigma^N)}{N} = H(f | \gamma),
\end{equation}
proving the entropic chaoticity and completing the proof.
\end{proof}
The last proof of this section will involve the second 'closeness' criteria, associated with the Fisher information functional, and given by Theorem \ref{thm: stability property via fisher}. The proof is similar to those appearing in \cite{HM} and \cite{Carr} with appropriate modifications. The proof will rely heavily on tools from the field of Optimal Transportation.
\begin{proof}[Proof of Theorem \ref{thm: stability property via fisher}]
The first step of the proof will be to show that conditions (\ref{eq: uniform nu_{mu_N}}) and (\ref{eq: condition of stability with fisher II}) imply that the marginal limit, $f$, satisfies the conditions of Theorem \ref{thm: chaoticity and entropic chaoticity general}. \\
We start by showing that $f\in L^p(\mathbb{R})$ for some $p>1$. In \cite{HM} the authors have presented a lower semi continuity result for the relative Fisher Information, from which we conclude that
\begin{equation}\label{eq: stability 0}
I(f | \gamma) \leq \liminf_{N\rightarrow\infty}\frac{I_N(\mu_N | \sigma^N)}{N} \leq C.
\end{equation}
Denoting by \begin{equation}\nonumber
I(f)=\int_{\mathbb{R}}\frac{\abs{f^\prime (x)}}{f(x)}dx =4 \int_{\mathbb{R}}\abs{\frac{d}{dx}\sqrt{f(x)}}^2 dx
\end{equation} 
we see that
\begin{equation}\nonumber
I(f) = I(f|\gamma)+2-\int_{\mathbb{R}}v^2f(v)dv < C+2-\int_{\mathbb{R}}v^2f(v)dv < \infty,
\end{equation}
as $f$ is a weak limit of $\Pi_1(\mu_N)$, implying that
\begin{equation}\nonumber
\int_{\mathbb{R}^2}v^2 f(v)dv \leq \liminf_{N\rightarrow\infty}\int_{\mathbb{R}}v^2 d\Pi_1(\mu_N)(v)=1.
\end{equation}
We conclude that $\sqrt{f}\in H^1(\mathbb{R})$ and using a Sobolev embedding theorem we find that $\sqrt{f}\in L^\infty(\mathbb{R})$. Thus, since $f$ is also in $L^1(\mathbb{R})$, we have that $f\in L^p(\mathbb{R})$ for all $p\geq1$.\\
The next step will be to show that condition (\ref{eq: uniform nu_{mu_N}}) implies a uniform bound for the $1+\alpha$ moment of $\Pi_1(\mu_N)$, i.e.
\begin{equation}\label{eq: uniform boundness of 1+alpha moment of Pi_1( mu_N)}
\int_{\mathbb{R}}\abs{v_1}^{1+\alpha}d\Pi_1(\mu_N)(v_1) \leq C,
\end{equation}
for some $C>0$, independent of $N$. This will show that
\begin{equation}\label{eq: conservation of second moment}
\int_{\mathbb{R}}v^2 f(v)dv = \lim_{N\rightarrow\infty}\int_{\mathbb{R}}v^2 d\Pi_1(\mu_N)(v)=1,
\end{equation}
as well as
\begin{equation}\label{eq: conservation of 1+alpha boundness}
\int_{\mathbb{R}}\abs{v}^{1+\alpha} f(v)dv \leq \liminf_{N\rightarrow\infty}\int_{\mathbb{R}}\abs{v}^{1+\alpha} d\Pi_1(\mu_N)(v)\leq C.
\end{equation}

To prove (\ref{eq: uniform boundness of 1+alpha moment of Pi_1( mu_N)}) we notice that
\begin{equation}\label{eq: unfirom 1+alpha moment I}
\begin{gathered}
\int_{\mathbb{R}}|v_1|^{1+\alpha}d\Pi(\mu_N)(v_1)=\frac{3-\alpha}{2^{3-\alpha}-1}\int_{\mathbb{R}} \int_{\frac{\abs{v_1}}{2}}^{\abs{v_1}}x^{\alpha-4}v_1^4 d\Pi_1(\mu_N)(v_1) dx\\
=\frac{3-\alpha}{2^{3-\alpha}-1}\int_0^\infty x^{\alpha-4} \pa{\int_{-2x}^{-x} v^4_1 d\Pi(\mu_N)(v_1) + \int_{x}^{2x}v^4_1 d\Pi(\mu_N)(v_1)} dv_1  dx \\
=\frac{3-\alpha}{2^{3-\alpha}-1}\int_0^\infty x^{\alpha-4} \pa{\int_{-2x}^{2x} v^4_1 d\Pi(\mu_N)(v_1) - \int_{-x}^{x}v^4_1 d\Pi(\mu_N)(v_1)} dv_1  dx \\
\end{gathered}
\end{equation}
Using condition (\ref{eq: uniform nu_{mu_N}}) we know that for any $\epsilon>0$ we can find $R>0$, such that for any $\abs{x}>R$ and any $N\in\mathbb{N}$
\begin{equation}\label{eq: unfirom 1+alpha moment II}
(1-\epsilon)C_S x^{2-\alpha} \leq \int_{-\sqrt{x}}^{\sqrt{x}}v_1^4 d\Pi_1(\mu_N)(v_1) \leq (1+\epsilon)C_S x^{2-\alpha}
\end{equation}
In addition, for any probability measure $\mu$ on $\mathbb{R}$ we have that
\begin{equation}\label{eq: unfirom 1+alpha moment III}
\int_{-x}^x v^4 d\mu(v) \leq 2x^4.
\end{equation}
Combining (\ref{eq: unfirom 1+alpha moment I}), (\ref{eq: unfirom 1+alpha moment II}) and (\ref{eq: unfirom 1+alpha moment III}) we conclude that
\begin{equation}\label{eq: unfirom 1+alpha moment IV}
\begin{gathered}
\int_{\mathbb{R}}|v_1|^{1+\alpha}d\Pi(\mu_N)(v_1) \leq  \frac{3-\alpha}{2^{3-\alpha}-1} \Bigg(32 R^{\frac{\alpha+1}{2}} \\
+ C_S\pa{(1+\epsilon)2^{4-2\alpha}-(1-\epsilon)}\int_{\sqrt{R}}^\infty \frac{dx}{x^\alpha} \Bigg) = C
\end{gathered}
\end{equation}
for a choice of $0<\epsilon<1$. \\
Lastly, we want to show that $\nu_f$, defined in Theorem \ref{thm: chaoticity and entropic chaoticity general}, satisfies the appropriate growth condition.\\
Since $\Pi_1(\mu_N)$ converges to $f$ weakly, we have that for any lower semi continuous function, $\phi$, that is bounded from below,
\begin{equation}\label{eq: conditions on pi_1(mu_N) imply nu_f is good I}
\int_{\mathbb{R}}\phi(v)f(v)dv \leq \liminf_{N\rightarrow\infty}\int_{\mathbb{R}}\phi(v_1)d\Pi_1(\mu_N)(v_1). 
\end{equation}
Similarly, if $\phi$ is upper semi continuous and bounded from above then 
\begin{equation}\label{eq: conditions on pi_1(mu_N) imply nu_f is good II}
\int_{\mathbb{R}}\phi(v)f(v)dv \geq \limsup_{N\rightarrow\infty}\int_{\mathbb{R}}\phi(v_1)d\Pi_1(\mu_N)(v_1). 
\end{equation}
Choosing $\phi(v)=v^4 \chi_{(-\sqrt{x},\sqrt{x})}(v)$ and $\phi(v)=v^4 \chi_{[-\sqrt{x},\sqrt{x}]}(v)$ respectively, and using condition (\ref{eq: uniform nu_{mu_N}}) proves that
\begin{equation}\nonumber
\nu_f(x) = \int_{-\sqrt{x}}^{\sqrt{x}}v^4 f(v)dv \underset{x\rightarrow\infty}{\sim}C_S x^{2-\alpha},
\end{equation}
and we can conclude that $f$ satisfies the conditions of Theorem \ref{thm: chaoticity and entropic chaoticity general}. This implies that the function $F_N=\frac{f^{\otimes N}}{\mathcal{Z}_N\pa{f,\sqrt{N}}}$ is well defined, and as usual we denote $\nu_N=F_N d\sigma^N$.\\
Next, we will show that $\frac{I_N(\nu_N |\sigma^N)}{N} $ is uniformly bounded in $N$. Denoting by $\nabla$ the normal gradient on $\mathbb{R}^N$ and by $\nabla_S$ its tangential component to Kac's sphere we find that
\begin{equation}\label{eq: stability fisher I}
\begin{gathered}
\int_{\mathbb{S}^{N-1}\pa{\sqrt{N}}}\frac{\abs{\nabla_S F_N}^2}{F_N}d\sigma^N \leq  \frac{1}{\mathcal{Z}_N\pa{f,\sqrt{N}}}\int_{\mathbb{S}^{N-1}\pa{\sqrt{N}}}\frac{\abs{\nabla f^{\otimes N}}^2}{f^{\otimes N}}d\sigma^N \\
=\sum_{i=1}^N \frac{1}{\mathcal{Z}_N\pa{f,\sqrt{N}}}\int_{\mathbb{S}^{N-1}\pa{\sqrt{N}}}\frac{\abs{ f^\prime (v_i)}^2}{f(v_i)}\Pi_{j=1,j\not=i}^N f(v_j)d\sigma^N \\
=N\int_{\mathbb{R}}\frac{\abs{\mathbb{S}^{N-2}}\pa{N-v_1^2}_+^\frac{N-3}{2}{}}{\abs{\mathbb{S}^{N-1}}N^{\frac{N-2}{2}}}\frac{\mathcal{Z}_{N-1}\pa{f, \sqrt{N-v_1^2}}}{\mathcal{Z}_N\pa{f,\sqrt{N}}}\cdot \frac{\abs{f^\prime (v_1)}^2}{f(v_1)}dv_1,
\end{gathered}
\end{equation}
where we have used Lemma \ref{lem: integration over sphere}, and the definition of the normalisation function. Using the asymptotic behaviour of $\mathcal{Z}_N\pa{f,\sqrt{r}}$ from Theorem \ref{thm: asyptotic behaviour for h conv N and normalisation function} we conclude that
\begin{equation}\label{eq: stability fisher II}
\begin{gathered}
\frac{I_N(\nu_N | \sigma^N)}{N} \leq \pa{\frac{N}{N-1}}^{\frac{1}{\alpha}} \int_{\mathbb{R}}\frac{\gamma_{\sigma,\alpha,1}\pa{\frac{1-v_1^2}{N^{\frac{1}{\alpha}}}}+\lambda_{N-1}\pa{N-v_1^2}}{\gamma_{\sigma,\alpha,1}(0)+\lambda_N(N)} \frac{\abs{f^\prime(v_1)}^2}{f(v_1)}dv_1 \\
\leq C I(f) \leq C_1,
\end{gathered}
\end{equation}
for $C_1 >0$, independently of $N$.\\
At this point we'd like to invoke the HWI inequality, a strategy that was first proved to be successful in this context in \cite{HM} and \cite{MM}. In our settings we find that 
\begin{equation}\label{eq: HWI }
\begin{gathered}
H(\mu_N | \sigma^N)-H(\nu_N | \sigma^N) \leq \frac{\pi}{2}\sqrt{I_N(\mu_N | \sigma_N)}W_2(\mu_N,\nu_N) \\
H(\nu_N | \sigma^N)-H(\mu_N | \sigma^N) \leq \frac{\pi}{2}\sqrt{I_N(\nu_N | \sigma_N)}W_2(\mu_N,\nu_N),
\end{gathered}
\end{equation}
where $W_2$ stands for the quadratic Wasserstein distance with distance function induced from the quadratic distance function on $\mathbb{R}^N$:
\begin{equation}\nonumber
W^2_2(\mu_N,\nu_N)=\inf_{\pi\in\Pi(\mu_N,\nu_N)}\int_{\mathbb{S}^{N-1}\pa{\sqrt{N}}\times \mathbb{S}^{N-1}\pa{\sqrt{N}}}\abs{x-y}^2 d\pi(x,y),
\end{equation}
where $\Pi(\mu_N,\nu_N)$, the space of pairing, is the space of all probability measures on $\mathbb{S}^{N-1}\pa{\sqrt{N}}\times \mathbb{S}^{N-1}\pa{\sqrt{N}}$ with marginal $\mu_N$ and $\nu_N$ respectively. \\
The reason we are allowed to use the HWI inequality follows from the fact that Kac's sphere has a positive Ricci curvature. Moreover, in the original statement of the HWI inequality, the quadratic Wasserstein distance is taken with the quadratic \emph{geodesic} distance, yet, fortunately for us, it is equivalent to the normal distance on $\mathbb{R}^N$, hence the factor $\frac{\pi}{2}$ that appears in (\ref{eq: HWI }). For more information about the Wasserstein distance and the HWI inequality, we refer the interested reader to \cite{Vtransport}.\\
 Combining (\ref{eq: HWI }) with the boundness of the rescaled relative Fisher information of $\mu_N$ and $\nu_N$ with respect to $\sigma^N$, we conclude that
\begin{equation}\label{eq: stability fisher III}
\abs{\frac{H(\mu_N | \sigma^N)}{N}-\frac{H(\nu_N | \sigma^N)}{N}} \leq C\frac{W_2(\mu_N,\nu_N)}{\sqrt{N}}
\end{equation}  
for some $C>0$.\\
The next step of the proof is to show that the first marginals of $\mu_N$ and $\nu_N$ have some joint bounded moment of order $l>2$, uniformly in $N$. This will help us give a quantitative estimation to the quadratic Wasserstein distance. Indeed, using several results from \cite{HM}, one can show the following estimation:
\begin{equation}\label{eq: stability fisher V}
\frac{W_2(\kappa_N,f^{\otimes N})}{\sqrt{N}} \leq C_1 B_l^{\frac{1}{l}}\pa{W_1\pa{\Pi_2(\kappa_N),f^{\otimes 2}}+\frac{1}{N^{p_1}}}^{\frac{1}{2}-\frac{1}{l}}
\end{equation}
where  $C_1$ and $p_1$ are positive constants that depends only on $l>2$, $\kappa_N$ is a probability measure on Kac's sphere, $f$ is a probability measure on $\mathbb{R}$ and
\begin{equation}\nonumber
B_l = \int_{\mathbb{R}}|v_1|^ld\Pi_1(\kappa_N)(v_1) + \int_{\mathbb{R}}|v_1|^l f(v_1)dv_1<\infty.
\end{equation} 
We have already shown that $\br{\Pi_1(\mu_N)}_{N\in\mathbb{N}}$ has a uniformly bounded moment of order $1+\alpha$. Using (\ref{eq: chaoticity and entropic chatoicity proof II}) from the proof of Theorem \ref{thm: chaoticity and entropic chaoticity general}, we find that
\begin{equation}\nonumber
\begin{gathered}
\int_{\mathbb{R}}|v_1|^{1+\alpha} d\Pi_1(\nu_N)(v_1)=\pa{\frac{N}{N-1}}^{\frac{1}{\alpha}}\int_{\abs{v_1}\leq \sqrt{N}}\frac{\gamma_{\sigma,\alpha,1}\pa{\frac{1-v_1^2}{N^{\frac{1}{\alpha}}}}+ \lambda_{N-1}\pa{N-v_1^2}}{\gamma_{\sigma,\alpha,1}(0)+ \lambda_N(N)}
\abs{v_1}^{1+\alpha}f(v_1)dv_1 
\end{gathered}
\end{equation}
for some $\sigma>0$, $1<\alpha<2$ and $\lambda_{N-k},\lambda_N$ with 
\begin{equation}\nonumber
\sup_u\abs{\lambda_{N-1}(u)}\underset{N\rightarrow\infty}{\longrightarrow}0, \quad\lambda_N(N)\underset{N\rightarrow\infty}{\longrightarrow}0.
\end{equation} 
Thus, along with (\ref{eq: conservation of 1+alpha boundness}), we conclude that
\begin{equation}\label{eq: stability fisher IV}
\int_{\mathbb{R}}|v_1|^{1+\alpha} d\Pi_1(\nu_N)(v_1) \leq C,
\end{equation}
for some $C>0$. \\
Defining 
\begin{equation}
\begin{gathered}
M = \int_{\mathbb{R}}|v_1|^{1+\alpha} d\Pi_1(\mu_N)(v_1)+\int_{\mathbb{R}}|v_1|^{1+\alpha} d\Pi_1(\nu_N)(v_1) \\
+ \int_{\mathbb{R}}|v_1|^{1+\alpha} f(v_1)dv_1 <\infty
\end{gathered}
\end{equation}
and combining (\ref{eq: stability fisher III}), (\ref{eq: stability fisher V})), and the triangle inequality for the Wasserstein distance, leads us to conclude that
\begin{equation}\label{eq: stability fisher VI}
\begin{gathered}
\abs{\frac{H(\mu_N | \sigma^N)}{N}-\frac{H(\nu_N | \sigma^N)}{N}} \leq C M^{\frac{1}{1+\alpha}}\Bigg[\pa{W_1\pa{\Pi_2(\mu_N),f^{\otimes 2}} +\frac{1}{N^{p_1}}}^{\frac{1}{2}-\frac{1}{1+\alpha}}\\
+\pa{W_1\pa{\Pi_2(\nu_N),f^{\otimes 2}}+\frac{1}{N^{p_1}}}^{\frac{1}{2}-\frac{1}{1+\alpha}}\Bigg].
\end{gathered}
\end{equation}
As $\Pi_2(\nu_N), \Pi_2(\nu_N)$ and $f^{\otimes 2}$ all have unit second moment (for any $N$), the Wasserstein distance is equivalent to weak topology with respect to them. Since $\br{\mu_N}_{N\in\mathbb{N}}$ and $\br{\nu_N}_{N\in\mathbb{N}}$ are $f-$chaotic, we conclude that
\begin{equation}\nonumber
W_1\pa{\Pi_2(\mu_N),f^{\otimes 2}}\underset{N\rightarrow\infty}{\longrightarrow}0, \quad W_1\pa{\Pi_2(\nu_N),f^{\otimes 2}}\underset{N\rightarrow\infty}{\longrightarrow}0,
\end{equation}
implying that
\begin{equation}\label{eq: stability fisher VII}
\lim_{N\rightarrow\infty}\abs{\frac{H(\mu_N | \sigma^N)}{N}-\frac{H(\nu_N | \sigma^N)}{N}}=0.
\end{equation}
We are almost ready to conclude the proof. Before we do, we use the lower semi continuity of the entropy, discussed in Theorem \ref{thm: CCRLV theorem}, to see that
\begin{equation}\nonumber
H(f|\gamma) \leq \liminf_{N\rightarrow\infty}\frac{H_N(\mu_N|\sigma^N)}{N} \leq C <\infty.
\end{equation}
Thus, 
\begin{equation}\label{eq: stability fisher VIII}
\begin{gathered}
\abs{\frac{H(\mu_N|\sigma^N)}{N}-H(f|\gamma)} \leq \abs{\frac{H(\mu_N | \sigma^N)}{N}-\frac{H(\nu_N | \sigma^N)}{N}} \\
+ \abs{\frac{H(\nu_N | \sigma^N)}{N}-H(f|\gamma)}\underset{N\rightarrow\infty}{\longrightarrow}0,
\end{gathered}
\end{equation}
where we have used (\ref{eq: stability fisher VII}) and Theorem \ref{thm: chaoticity and entropic chaoticity general}, completing proof.
\end{proof}
\begin{remark}
We'd like to point out that following the above proof, one can see that condition (\ref{eq: uniform nu_{mu_N}}), giving us a uniform asymptotic behaviour for the fourth moments of the first marginals of $\br{\mu_N}_{N\in\mathbb{N}}$, can be replaced with the conditions that $f$ satisfies the conditions of Theorem \ref{thm: chaoticity and entropic chaoticity general}, and the first marginals of $\br{\mu_N}_{N\in\mathbb{N}}$ have a uniformly bounded $k-$th moment, for some $k>2$. This gives us a different approach to the stability problem, expressed with the Fisher information functional, one that assumes less information on the first marginals, but more conditions on the marginal limit. 
\end{remark}

\section{Final Remarks.}\label{sec: final remarks}
While Kac's model, chaoticity and entropic chaoticity, and the many body Cercignani's conjecture are far from being completely understood and resolved, we hope that our paper has shed some light on the interplay between the moments of a generating function and its associated tensorised measure, restricted to Kac's sphere. As an epilogue, we present here a few remarks about our work, along with associated questions we'll be interested in investigating next.
\begin{itemize}
\item One fundamental problem we're very interested in is finding conditions under which the many body Cercignani's conjecture is valid. While our work showed that the requirement of a bounded fourth moment is not a major issue for chaoticity and even entropic chaoticity, we still believe that the fourth moment plays an important role in the conjecture. At the very least, due to its probabilistic interpretation as a measurement of deviation from the sphere, we believe that the fourth moment will be needed for an initial positive answer to the conjecture.
\item The following was communicated to us by Cl\'ement Mouhot: Using a Talagrand inequality, one can show that if the family of functions $\br{G_N}_{N\in\mathbb{N}}$, restricted to the sphere, satisfies a Log-Sobolev inequality that is uniform in $N$, one has that
\begin{equation}\nonumber
\lim_{N\rightarrow\infty}\frac{H(F_N|G_N)}{N}=0
\end{equation}
implies that $\lim_{N\rightarrow\infty}\pa{\Pi_k(F_N)-\Pi_k(G_N)}=0$. Our stability result, Theorem \ref{thm: stability property}, gives many examples where the function $G_N$ doesn't satisfy any Log-Sobolev inequality (due to how the underlying function behaves), but we still get equality of marginal. Moreover, we actually get that $F_N$ is entropically chaotic! The connection between the limit of the 'distance' 
\begin{equation}\nonumber
d(F_N,G_N)=\frac{H(F_N|G_N)}{N}
\end{equation}
and the convergence of marginals is still not understood fully.
\item We'll be interested to know if one can find an easy criteria for which we can evaluate quantitatively the convergence of $h^{\ast N}$ (appearing in Theorem \ref{thm: asyptotic behaviour for h conv N and normalisation function}) without relying on the reminder function. This will allow for possibilities to extend the work done by the second author in \cite{Einav1,Einav2} and allow the underlying generating function, $f$, to rely on $N$ as well. While we present such quantitative estimation in the Appendix, we found them to be unusable while trying to deal with concrete examples.
\end{itemize}

\appendix
\section{Additional Proofs.}
In this section of the appendix we will present several proofs of technical items we thought would only hinder the flow of the paper.
\begin{proof}[Proof of Lemma \ref{lem: technical FDA NDA lemma}]
Assume that the conclusion is false. We can find a sequence $x_n \underset{n\rightarrow\infty}{\longrightarrow} 0$, $x_n\not=0$, and an $\epsilon_0>0$ such that 
\[|g(x_n)|\geq \epsilon_0.\]
Due to continuity, we can find $d_1>0$ such that for any $x\in [x_1,x_1+d_1]$ we have
\[|g(x)|\geq \frac{\epsilon_0}{2}.\]
Denote $n_1=1$, $x_{k_1}=x_1$ and $\xi_1=n_1 \cdot x_1 = x_1$.\\
Since $x_n$ converges to zero and is non zero, we can find $x_{k_2}$ such that $0<x_{k_2}<\frac{\xi_1}{2}$. Let $n_2= \left[ \frac{\xi_1}{x_{k_2}}\right]+1\geq 2$, where $[\cdot]$ is the lower integer part function. We may assume that $x_{k_2}<d_1$ and conclude that
\begin{equation}\nonumber
\xi_1 \leq n_2 x_{k_2} < \xi_1 + x_{k_2} \leq \xi_1+n_1d_1. 
\end{equation}
Next, we can find $d_2$ such that $n_2(x_{k_2}+d_2) \leq \xi_1+n_1d_1$. We may also assume that $d_2$ is small enough so that $x\in[x_{k_2},x_{k_2}+d_2]$ implies 
\[|g(x)|\geq \frac{\epsilon_0}{2}.\]
Denoting by $\xi_2=n_2 x_{k_2}$, we notice that $[\xi_2,\xi_2+n_2 d_2]\subset [\xi_1,\xi_1+n_1d_1]$ and the closed intervals are non empty. \\
We continue by induction. Assume we found $n_i,k_i\in\mathbb{N}$, $n_i \geq i$, and $d_i>0$ for $i=1,\dots,j$ such that $\xi_i=n_i x_{k_i}$ satisfies
\[ [\xi_i,\xi_i + n_i d_i] \subset [\xi_{i-1}, \xi_{i-1}+n_{i-1}d_{i-1}]\]
and for any $x\in[\xi_i,\xi_i+n_id_i]$ we have that
\[\left\lvert g\left(\frac{x}{n_i} \right)\right\rvert\geq \frac{\epsilon_0}{2}.\]
We find $x_{k_{j+1}}$ such that $x_{k_{j+1}}<\frac{\xi_j}{j+1}$ and define $n_j=\left[\frac{\xi_j}{x_{k_{j+1}}} \right]+1 \geq j+1$. As such, we have that
\[\xi_j \leq n_{j+1}x_{k_{j+1}}< \xi_j + x_{k_{j+1}}< \xi_j +n_j d_j,\]
where the last inequality is valid since we can pick $x_{k_{j+1}}<n_j d_j$. We can find $d_{j+1}$ such that $n_{j+1}(x_{k_{j+1}}+d_{j+1}) < \xi_j +n_j d_j$ and for any $x\in [x_{k_{j+1}},x_{k_{j+1}}+d_{j+1}]$ 
\[|g(x)| \geq \frac{\epsilon_0}{2}.\]
Denoting $\xi_{j+1}=n_{j+1}x_{k_{j+1}}$ gives us the interval with the desired properties.\\
Since we have a nested sequence of non-empty closed intervals in $\mathbb{R}$ we know that the intersection of all of them must be non-empty. Thus, there exists $x\in [\xi_i,\xi_i+n_i d_i]$ for all $i\in\mathbb{N}$. Moreover, by construction
\[\left\lvert g\left(\frac{x}{n_i} \right)\right\rvert \geq \frac{\epsilon_0}{2}\]
which contradicts the assumption that $\lim_{n\rightarrow\infty}g\left(\frac{x}{n} \right)=0$ for any $x\not=0$.
\end{proof}
The next result we will prove, is  Lemma \ref{lem: g in NDA implies g is exponential at low frequencies}:
\begin{proof}[Proof of Lemma \ref{lem: g in NDA implies g is exponential at low frequencies}]
Since $\widehat{g}$ is in the NDA of $\widehat{\gamma}_{\sigma,\alpha,\beta}$ we conclude that $\widehat{g}$ is actually in the FDA of $\widehat{\gamma}_{\sigma,\alpha,\beta}$, due to Theorem \ref{thm: FDA and NDA}. Thus, there exists $\eta_1$, with $\frac{\eta_1(\xi)}{\abs{\xi}^{\alpha}}\in L^\infty(\mathbb{R})$ and 
\[\frac{\eta_1(\xi)}{\abs{\xi}^{\alpha}}\underset{\xi\rightarrow 0}{\longrightarrow}0,\]
such that
\[\widehat{g}(\xi)=1-\sigma \abs{\xi}^\alpha \pa{1+i\beta \text{sgn}(\xi)\tan\pa{\frac{\pi\alpha}{2}}}+\eta_1(\xi)\]
\[=e^{-\sigma \abs{\xi}^\alpha \pa{1+i\beta \text{sgn}(\xi)\tan\pa{\frac{\pi\alpha}{2}}}}+\eta_2(\xi)+\eta_1(\xi),\]
where $\eta_2(\xi)$ has the same properties as $\eta_1(\xi)$. We conclude that
\[\abs{\widehat{g}(\xi)} \leq e^{-\sigma \abs{\xi}^{\alpha}}+\abs{\eta_1(\xi)+\eta_2(\xi)}\leq 1-\sigma \abs{\xi}^\alpha + \abs{\eta_1(\xi)}+\abs{\eta_2(\xi)}+\abs{\eta_3(\xi)},\]
where $\eta_3(\xi)$ has the same properties as $\eta_1(\xi)$.\\
 Let $\beta_0>0$ be such that if $\abs{\xi}<\beta_0$ 
\[\abs{\eta_1(\xi)}+\abs{\eta_2(\xi)}+\abs{\eta_3(\xi)} \leq \frac{\sigma \abs{\xi}^\alpha}{2}.\]
For any $\abs{\xi}<\beta_0$ one has that
\[\abs{\widehat{g}(\xi)} \leq 1-\frac{\sigma \abs{\xi}^\alpha}{2} \leq e^{-\frac{\sigma\abs{\xi}^\alpha}{2}},\]
completing the proof.
\end{proof}

\section{Quantitative Approximation Theorem.}\label{app: quantitative approximation theorem}
An item of great importance in Kinetic Theory, and our problem in particular, is \emph{quantitative} estimation of errors. Our local L\'evy Central Limit Theorem involves such an estimation, yet it is dependent on the function
\begin{equation}\nonumber
\omega(\beta)=\sup_{\abs{\xi}}\frac{\abs{\eta(\xi)}}{\abs{\xi}^{\alpha}},
\end{equation}
where $\eta$ is the reminder function of a probability density function $g$ in the NDA of some $\widehat{\gamma}_{\sigma,\alpha,\beta}$. In some cases one can find explicit estimation for the behaviour of $\eta$ near zero, and get a better quantitative estimation on the error term $\epsilon(N)$. Such conditions are explored in \cite{GJT} and we will satisfy ourselves by mentioning them, but providing no proof. \\
\begin{definition}\label{def: FDA of order delta}
Let $\delta>0$. The Fourier Domain of Attraction \emph{of order} $\delta$ of $\widehat{\gamma}_{\sigma,\alpha,\beta}$ is the subset of the FDA of $\widehat{\gamma}_{\sigma,\alpha,\beta}$ such that the reminder function, $\eta$, satisfies
\begin{equation}\nonumber
\frac{\abs{\eta(\xi)}}{\abs{\xi}^{\alpha}}\leq C \abs{\xi}^\delta,
\end{equation}
for some $C>0$.
\end{definition}
Clearly the FDAs of order $\delta$ are nested sets, all contained in the $FDA$. Also, if $g$ is in the FDA of order $\delta$ of $\widehat{\gamma}_{\sigma,\alpha,\beta}$ then we can replace $\omega(\beta)$, defined in Theorem \ref{thm: main approximation themroem} by $C\beta^{\delta}$ and get an explicit estimation to the error term $\epsilon(N)$!.\\
The following is a variant of a theorem appearing in \cite{GJT} that gives sufficient conditions to be in the FDA of order $\delta$ of some $\widehat{\gamma}_{\sigma,\alpha,\beta}$:
\begin{theorem}\label{thm: appendix FDA of order delta}
Let $g$ be a probability density on $\mathbb{R}$ that has zero mean. Let $1<\alpha<2$ and $0<\delta<2-\alpha$ be given. Then if
\begin{equation}\label{eq: appendix FDA of order delta}
\int_{\mathbb{R}}\abs{x}^{\alpha+\delta}\abs{g(x)-\gamma_{\sigma,\alpha,\beta}(x)}dx < \infty
\end{equation}
for some $\sigma>0$ and $\beta\in [-1,1]$, $g$ is in the FDA of order $\delta$ of $\widehat{\gamma}_{\sigma,\alpha,\beta}$.
\end{theorem}

\end{document}